\documentclass[11pt]{amsart}
\usepackage{amsmath,amssymb,amsthm}
\usepackage[margin=1.1in]{geometry}
\usepackage[colorlinks=true,linkcolor=blue,citecolor=blue,urlcolor=blue]{hyperref}

\newtheorem{theorem}{Theorem}[section]
\newtheorem{lemma}[theorem]{Lemma}
\newtheorem{proposition}[theorem]{Proposition}
\newtheorem{corollary}[theorem]{Corollary}
\theoremstyle{definition}
\newtheorem{definition}[theorem]{Definition}
\newtheorem{problem}{Problem}
\theoremstyle{remark}
\newtheorem{remark}[theorem]{Remark}

\newcommand{\R}{\mathbb{R}}
\newcommand{\N}{\mathbb{N}}
\newcommand{\Q}{\mathbb{Q}}
\newcommand{\Z}{\mathbb{Z}}
\newcommand{\F}{\mathcal{F}}
\newcommand{\M}{\mathcal{M}}
\newcommand{\RP}{\mathbb{RP}^1}
\newcommand{\Img}{\operatorname{Im}}
\newcommand{\sgp}[1]{\langle #1\rangle}
\newcommand{\ol}{\overline}

\title[Transitive semigroups of linear fractional maps]{Topologically transitive semigroup actions\\ of real linear fractional transformations}
\author{Mohammad Javaheri}
\address{Department of Mathematics and Statistics, Siena University\\
515 Loudon road, Loudonville, NY 12211, USA}
\email{mjavaheri@siena.edu}
\date{Revised version}

\begin{document}

\begin{abstract}
Let $\F_I=\{f:I\to I \mid f(x)=(Ax+B)/(Cx+D),\ AD-BC\neq 0\}$, where $I$ is a proper closed subinterval of $\R$. We describe all pairs $f,g\in\F_I$ for which the action of the semigroup $\sgp{f,g}$ generated by $f$ and $g$ is topologically transitive on $I$. It turns out that for such pairs topological transitivity, hypercyclicity, and density of every interior orbit are equivalent. Using the Hilbert metric of $I$ we show that no subsemigroup of $\F_I$, with any number of generators, acts in a weakly topologically mixing way on $I$. We classify the pairs of affine maps of $\R$ whose semigroup action is topologically transitive, weakly topologically mixing, or hypercyclic, and we show that such an action is never $3$-transitive. In contrast, there are pairs of real linear fractional transformations whose semigroup action on $\R$ is topologically $3$-transitive, while no semigroup of linear fractional transformations is $4$-transitive.
\end{abstract}

\maketitle

\section{Introduction}

Let $X$ be a Banach space (over $\R$ or $\mathbb{C}$) and $T$ a bounded linear operator on $X$. We say $T$ is \emph{hypercyclic} if there exists $x\in X$ such that $\{T^nx:n\in\N\}$ is dense in $X$. The general theory of hypercyclicity was initiated by Rolewicz \cite{Rol}, who gave the first examples of hypercyclic operators on Banach spaces. It is well known that no finite-dimensional space admits a hypercyclic operator (however, a pair of linear operators may generate a hypercyclic semigroup in the sense defined below; see \cite{J2}). On the other hand, Ansari, Bernal-Gonz\'alez, and Bonet--Peris \cite{Ans,Ber,BP} proved that every separable infinite-dimensional Fr\'echet space admits a hypercyclic operator. For surveys of hypercyclicity see \cite{GE1,GE2}.

More generally, a semigroup $G$ of maps of a topological space $X$ is called \emph{hypercyclic} if there exists $x\in X$ such that its orbit $\{g(x):g\in G\}$ is dense in $X$. Hypercyclic continuous semigroups of bounded linear operators have been studied by several authors; see for example \cite{CMP,CP,DSW,Kal} and the survey \cite{BMP}.

In this paper we are interested in semigroups generated by two maps from the set of real linear fractional transformations
\begin{equation}\label{eq:F}
\F=\left\{\frac{Ax+B}{Cx+D}:\ AD-BC\neq 0\right\}.
\end{equation}
Regarded as homeomorphisms of the projective line $\RP=\R\cup\{\infty\}$, the elements of $\F$ form a group under composition which is isomorphic to $PGL_2(\R)\cong PSL_2^{\pm}(\R)$, the group of real $2\times 2$ matrices of determinant $\pm1$ modulo $\pm\mathrm{Id}$. To see the relation with linear dynamics, let
\begin{equation}\label{eq:M}
\M=\left\{\begin{pmatrix}A&B\\ C&D\end{pmatrix}:\ AD-BC\neq0\right\},
\end{equation}
and let $\Phi:\M\to\F$ be the obvious map. Then $\Phi$ is a homomorphism, $\Phi(M_1M_2)=\Phi(M_1)\circ\Phi(M_2)$. If $\pi:\R^2\setminus\{y=0\}\to\R$ is given by $\pi(x,y)=x/y$, then $\pi\circ M=\Phi(M)\circ\pi$ wherever both sides are defined. Since $\pi$ is continuous, open and surjective, if a semigroup in $\M$ is hypercyclic on $\R^2$ then its image under $\Phi$ has dense orbits in $\R$; this is our motivation for studying semigroups generated by maps in $\F$.

For a proper closed interval $I\subsetneq\R$ let
\[
\F_I=\{f\in\F:\ f \text{ has no pole in } I \text{ and } f(I)\subseteq I\}.
\]
A single map $f\in\F_I$ has a simple dynamics: since $f$ maps $I$ into itself it has a fixed point in $I$, so it is not elliptic, and the orbit of every $x\in I$ is either periodic of period at most $2$ or converges to a fixed point (possibly at infinity). Consequently a single map in $\F_I$ is never hypercyclic, while a pair of maps can be. We point out that the restriction to $\F_I$ is essential here: an elliptic element of $\F$ with irrational rotation number has dense orbits in $\RP$ (and in $\R$); see Section~\ref{sec:mobR}.

For $f,g:I\to I$ let $\sgp{f,g}$ denote the semigroup generated by $f$ and $g$, and for $x\in I$ let $O(x)=\{h(x):h\in\sgp{f,g}\}$ be the orbit of $x$ and $\Omega_x=\ol{O(x)}$ its closure. Note that $\Omega_x$ is invariant under $f$ and $g$. If $O(x)$ is dense for some $x$ we call $(f,g)$ a \emph{hypercyclic pair}. An important example on $I=(0,\infty)$, related to continued fractions, is the pair $f(x)=x+1$, $g(x)=1/x$: the orbit of $1$ is the set of all positive rationals. Other examples are $f(x)=ax$, $g(x)=bx+c$ on $[0,\infty)$ with $0<a<1\le b$ and $c>0$ (Lemma~\ref{lem:affine}; see also \cite{BMS,J1}).

It is easy to construct hypercyclic semigroups of maps of an interval. We call $f:I\to I$ \emph{length-decreasing} if $|f(J)|<|J|$ for every nondegenerate subinterval $J\subseteq I$, where $|J|$ denotes length; $\Img(f)$ denotes the image of $f$.

\begin{theorem}\label{thm:length}
Let $\{f_i:I\to I\mid i\in\Lambda\}$ be a family of continuous length-decreasing maps of a compact interval $I$ ($\Lambda$ arbitrary). Suppose that for each $i$ the maximum and minimum of $f_i$ on $I$ are attained at the endpoints of $I$, and that
\begin{equation}\label{eq:cover}
\bigcup_{i\in\Lambda}\Img(f_i)=I.
\end{equation}
Then the orbit of every $x\in I$ under the semigroup generated by the $f_i$ is dense in $I$.
\end{theorem}

Besides hypercyclicity, the following notions are standard in topological dynamics of semigroup actions (we follow \cite{CKN}).

\begin{definition}\label{def:trans}
Let $G$ be a semigroup acting by continuous maps on a topological space $X$. The action is
\begin{enumerate}
\item[(a)] \emph{topologically transitive} if for all nonempty open $U,V\subseteq X$ there is $g\in G$ with $g(U)\cap V\neq\emptyset$;
\item[(b)] \emph{topologically $k$-transitive} if the diagonal action of $G$ on $X^k$ is topologically transitive. Topological $2$-transitivity is also called \emph{weak topological mixing}.
\end{enumerate}
\end{definition}

For results on transitivity and mixing of operators and semigroups see \cite{BBCP,CS}. If $X$ is a second countable Baire space, topological transitivity implies hypercyclicity (Lemma~\ref{lem:baire}); conversely, if there is a dense set of points with dense orbits, the action is topologically transitive.

A point $\theta\in I$ is an \emph{attracting fixed point} of $f\in\F_I$ if $f(\theta)=\theta$ and $f^n(x)\to\theta$ for all $x\in I$ sufficiently close to $\theta$ (we allow parabolic fixed points, which may attract only from one side; at an endpoint of $I$ only one side matters). A non-identity $f\in\F_I$ has at most one attracting fixed point, which we denote by $o(f)$; writing $o(f)$ implicitly asserts its existence. If $I$ is unbounded we allow $o(f)=\pm\infty$, meaning $f^n(x)\to\pm\infty$ for all $x$ near $\pm\infty$. We write $\partial I$ for the set of endpoints of $I$, possibly including $\pm\infty$. Our main result is the following.

\begin{theorem}\label{thm:main}
Let $f,g\in\F_I$, where $I$ is a proper closed subinterval of $\R$. Then the orbit of a given $x\in I\setminus\partial I$ under $\sgp{f,g}$ is dense in $I$ exactly in the following cases:
\begin{enumerate}
\item[I.] Both $f$ and $g$ are increasing and one of the following holds:
 \begin{enumerate}
 \item[(i)] at least one of $f,g$ is not onto, $\Img(f)\cup\Img(g)=I$, and $\{o(f),o(g)\}=\partial I$;
 \item[(ii)] there exist $a,b>0$ with $\min\{a,b\}<1<\max\{a,b\}$ and $\ln a/\ln b\notin\Q$, and a linear fractional bijection $\theta:[0,\infty]\to I$, such that $\theta^{-1}\circ f\circ\theta(x)=ax$ and $\theta^{-1}\circ g\circ\theta(x)=bx$.
 \end{enumerate}
\item[II.] Exactly one of $f,g$, say $f$, is increasing, $\{o(f),g(o(f))\}=\partial I$, and one of the following holds:
 \begin{enumerate}
 \item[(i)] $g$ is not onto and $\Img(f)\cup\Img(g)=I$;
 \item[(ii)] $g$ is onto, $f$ is not onto, and $\Img(fg)\cup\Img(gf)=I$.
 \end{enumerate}
\item[III.] Both $f$ and $g$ are decreasing, $\{o(fg),o(gf)\}=\partial I$, and one of the following holds:
 \begin{enumerate}
 \item[(i)] neither $f$ nor $g$ is onto and $\Img(f)\cup\Img(g)=I$;
 \item[(ii)] exactly one of $f,g$ is onto and $\Img(fg)\cup\Img(gf)=I$.
 \end{enumerate}
\end{enumerate}
\end{theorem}

See Theorem~\ref{thm:summary} for an equivalent formulation in normal forms. The conditions in Theorem~\ref{thm:main} do not depend on $x$; hence, if one interior orbit is dense then all interior orbits are dense. Together with Lemma~\ref{lem:baire} this gives:

\begin{corollary}\label{cor:equiv}
Let $f,g\in\F_I$ with $I$ a proper closed interval. The following are equivalent: (a) the action of $\sgp{f,g}$ on $I$ is topologically transitive; (b) it is hypercyclic; (c) the orbit of every $x\in I\setminus\partial I$ is dense; (d) one of the cases of Theorem~\ref{thm:main} holds.
\end{corollary}

This equivalence fails for affine pairs on $\R$, where hypercyclicity does not imply transitivity (Theorems~\ref{thm:affine-trans} and \ref{thm:affine-hyp}). For weak mixing we prove the following, which holds for semigroups with an arbitrary number of generators.

\begin{theorem}\label{thm:nomix}
Let $I\subsetneq\R$ be a proper interval and let $G\subseteq\F_I$ be any semigroup. Then the action of $G$ on $I$ is not weakly topologically mixing.
\end{theorem}

Finally, on the whole line the situation is different (Section~\ref{sec:mobR}): there are $f,g\in\F$ whose semigroup action on $\R$ is topologically $3$-transitive, while no semigroup in $\F$ is $4$-transitive (Theorem~\ref{thm:3trans}), and no semigroup of affine maps is $3$-transitive (Proposition~\ref{prop:nk}).

The paper is organized as follows. In Section~\ref{sec:inc} we prove Theorem~\ref{thm:length}, collect preliminary lemmas, and prove case I of Theorem~\ref{thm:main}. Cases II and III are treated in Sections~\ref{sec:mixed} and \ref{sec:dec}. In Section~\ref{sec:mix} we prove Theorem~\ref{thm:nomix} and study affine pairs on $\R$, and in Section~\ref{sec:mobR} we discuss pairs in $\F$ acting on $\R$. Section~\ref{sec:concl} contains a summary and open questions.

\section{Preliminaries and Case I: both maps increasing}\label{sec:inc}

\begin{proof}[Proof of Theorem~\ref{thm:length}]
Let $I=[a,b]$ and $\Omega=\Omega_x$.

\emph{Step 1: $a,b\in\Omega$.} We first note: if $\varphi:I\to I$ is continuous, length-decreasing and $\varphi(a)=a$, then $\varphi^n(y)\to a$ for all $y\in I$. Indeed, $J_n=\varphi^n([a,y])$ is an interval containing $a$, say $J_n=[a,m_n]$, and $|J_1|<|J_0|$ gives $J_1\subseteq J_0$, hence inductively $J_{n+1}\subseteq J_n$. If $m=\lim m_n>a$, then by continuity $\max_{[a,m]}\varphi=\lim\max_{[a,m_n]}\varphi=\lim m_{n+1}=m$, while $\min_{[a,m]}\varphi=a$; thus $|\varphi([a,m])|=|[a,m]|$, a contradiction. Hence $m=a$ and $\varphi^n(y)\in J_n\to\{a\}$. The same holds with $b$ in place of $a$.

Choose $i,j\in\Lambda$ with $a\in\Img(f_i)$ and $b\in\Img(f_j)$. Since extrema are attained at endpoints, $a\in\{f_i(a),f_i(b)\}$ and $b\in\{f_j(a),f_j(b)\}$. If $f_i(a)=a$ and $f_j(b)=b$, then $a,b\in\Omega$ by the observation above. If $f_i(b)=a$ and $f_j(b)=b$, then $b\in\Omega$ and $a=f_i(b)\in\Omega$; the case $f_i(a)=a$, $f_j(a)=b$ is symmetric. If $f_i(b)=a$ and $f_j(a)=b$, then $\varphi=f_i\circ f_j$ and $\psi=f_j\circ f_i$ are continuous, length-decreasing, and $\varphi(a)=a$, $\psi(b)=b$, so again $a,b\in\Omega$.

\emph{Step 2.} Suppose $\Omega\neq I$. Then $I\setminus\Omega$ is a nonempty union of disjoint open intervals contained in $(a,b)$; since their total length is finite, one of them, $A=(s,t)$, has maximal length, and $s,t\in\Omega$. Since $f_i(a),f_i(b)\in\Omega$, the endpoints of the interval $\Img(f_i)$ are in $\Omega$ for every $i$. By \eqref{eq:cover}, $A$ meets some $\Img(f_i)$, and since $A$ is connected and contains no endpoint of $\Img(f_i)$, we get $A\subseteq\Img(f_i)$. Among all pairs $(p,q)\in I^2$ with $f_i(p)=s$, $f_i(q)=t$ (a nonempty compact set) choose one with $|p-q|$ minimal, and let $B$ be the open interval between $p$ and $q$. If $f_i(r)\le s$ for some $r\in B$, the intermediate value theorem gives $r'$ between $r$ and $q$ with $f_i(r')=s$, contradicting minimality; similarly $f_i(r)\ge t$ is impossible. Hence $f_i(B)=A$. Since $\Omega$ is $f_i$-invariant, $B\cap\Omega=\emptyset$, so $|B|\le|A|$ by maximality of $A$. But $|A|=|f_i(\ol B)|<|\ol B|=|B|$, a contradiction.
\end{proof}

The following lemma describes the hypercyclic pairs of affine maps on $[0,\infty)$. We include a complete proof.

\begin{lemma}\label{lem:affine}
Let $f(x)=ax$ and $g(x)=bx+c$ with $a,b>0$ and $c\ge0$. Then the orbit of a given $x>0$ under $\sgp{f,g}$ is dense in $[0,\infty)$ if and only if one of the following holds:
\begin{enumerate}
\item[(i)] $c>0$ and $0<a<1\le b$;
\item[(ii)] $c=0$, $\min\{a,b\}<1<\max\{a,b\}$, and $\ln a/\ln b\notin\Q$.
\end{enumerate}
\end{lemma}

\begin{proof}
\emph{Sufficiency of (ii).} $O(x)=\{a^mb^nx: m,n\ge0,\ m+n\ge1\}$, and since $\ln a<0<\ln b$ (or vice versa) with irrational ratio, the additive semigroup generated by $\ln a,\ln b$ is dense in $\R$. Hence $O(x)$ is dense.

\emph{Sufficiency of (i).} Conjugating by $x\mapsto x/c$ we may assume $c=1$. Since $f^n(x)\to0$ and $\Omega_x$ is closed and invariant, $\Omega_x\supseteq\Omega_0$, so it suffices to show that $O(0)$ is dense.

Suppose $b=1$. The orbit of $0$ contains all numbers $\sum_{i=0}^kn_ia^i$ with $n_i,k\in\N\cup\{0\}$ (e.g.\ $g^{n_0}fg^{n_1}(0)=n_0+n_1a$). Given $y\ge0$, let $n_0=\lfloor y\rfloor$ and, inductively, let $n_{k+1}$ be the integer part of $a^{-k-1}(y-\sum_{i=0}^kn_ia^i)$. Then $0\le y-\sum_{i=0}^kn_ia^i<a^k$ for all $k$, so $y\in\Omega_0$.

Suppose $b>1$ and $\ln a/\ln b\notin\Q$. For $m,n\ge1$ we have $f^mg^n(0)=a^m(b^n-1)/(b-1)$. Given $\lambda>0$ there are $m_k,n_k\to\infty$ with $a^{m_k}b^{n_k}\to\lambda$, and then $f^{m_k}g^{n_k}(0)\to\lambda/(b-1)$. Hence $O(0)$ is dense.

Suppose $b>1$ and $\ln a/\ln b=-p/q$ with $p,q\in\N$, so that $a^qb^p=1$. Let $s=a^q\in(0,1)$. For $k\ge1$ the map $\varphi_k=f^{qk}g^{pk}$ is a translation:
\[
\varphi_k(x)=a^{qk}\Big(b^{pk}x+\frac{b^{pk}-1}{b-1}\Big)=x+\tau_k,\qquad \tau_k=\frac{1-s^k}{b-1}.
\]
Since compositions of translations add, $O(0)$ contains the additive semigroup $S$ generated by $\{\tau_k\}$. Let $\varepsilon>0$, choose $k$ with $e:=\tau_{k+1}-\tau_k=s^k(1-s)/(b-1)<\varepsilon$, and put $u=\tau_{k+1}$, $v=\tau_k$. For $N\ge1$ and $0\le j\le N$ we have $(N-j)v+ju=Nv+je\in S$; these points are $e$-dense in $[Nv,Nu]$. If $N\ge N_0:=\lceil v/e\rceil$ then $Nu\ge(N+1)v$, so the intervals $[Nv,Nu]$, $N\ge N_0$, cover $[N_0v,\infty)$, and $S$ is $\varepsilon$-dense there. Applying $f^m$, the set $f^m(S)\subseteq O(0)$ is $a^m\varepsilon$-dense in $[a^mN_0v,\infty)$. Since $\varepsilon$ is arbitrary and $a^m\to0$, $O(0)$ is dense in $[0,\infty)$.

\emph{Necessity.} Let $c>0$. If $a,b\ge1$ then $f(y)\ge y$ and $g(y)\ge y$, so $O(x)\subseteq[x,\infty)$. If $b<1$ and $a\le1$, then $O(x)\subseteq[0,\max\{x,c/(1-b)\}]$. If $b<1<a$, then $f(y)\ge y$ and $g(y)\ge c$, so $O(x)\subseteq[\min\{x,c\},\infty)$. Thus (i) is necessary. Let $c=0$. If $a,b\le1$ or $a,b\ge1$ the orbit is bounded or bounded away from $0$. If $\ln a/\ln b\in\Q$ then $\{a^mb^n\}\subseteq r^{\Z}$ for some $r>1$, so $O(x)$ is discrete in $(0,\infty)$.
\end{proof}

\begin{remark}\label{rem:conj}
From now on we mostly work on $I=[0,1]$ and write $\F$ for $\F_{[0,1]}$ when no confusion arises. For any other proper closed interval $J$ there is a linear fractional bijection $\theta:J\to[0,1]$ (with $\infty$ mapped to an endpoint if $J$ is unbounded), e.g.\ $\theta(x)=(x-a)/(b-a)$ for $J=[a,b]$, $\theta(x)=1/(x-a+1)$ for $J=[a,\infty)$, and $\theta(x)=1/(b-x+1)$ for $J=(-\infty,b]$. The transformation $(f,g)\mapsto(\theta f\theta^{-1},\theta g\theta^{-1})$ (\emph{conjugation by $\theta$}) maps dense orbits to dense orbits and attracting fixed points to attracting fixed points, and all statements of Theorem~\ref{thm:main} are invariant under such conjugations. We use this freely; in particular the change of variable $x\mapsto1-x$.
\end{remark}

For $u<1$ let
\begin{equation}\label{eq:thetau}
\theta_u(x)=\frac{x}{ux+1-u}.
\end{equation}
Each $\theta_u$ is an increasing linear fractional bijection of $[0,1]$ fixing $0$ and $1$, hence conjugation by $\theta_u$ preserves the multipliers at $0$ and $1$. Its pole $1-1/u$ ranges over $(-\infty,0)$ for $u\in(0,1)$, equals $\infty$ for $u=0$, and ranges over $(1,\infty)$ for $u<0$; thus every point of $\RP\setminus[0,1]$ is the pole of exactly one $\theta_u$.

\begin{lemma}[Normal forms]\label{lem:normal}
\begin{enumerate}
\item[(a)] If $f\in\F$ is increasing with $f(0)=0$ and $f(1)<1$, then there is $u<1$ such that $\theta_uf\theta_u^{-1}(x)=x/(x+a)$ with $a=1/f'(0)>0$.
\item[(b)] If $g\in\F$ is decreasing with $g(0)=1$ and $g(1)>0$, then there is $u<1$ such that $\theta_ug\theta_u^{-1}(x)=a/(x+a)$ for some $a>0$.
\end{enumerate}
\end{lemma}

\begin{proof}
(b) Regard $g$ as a homeomorphism of $\RP$ and let $z=g^{-1}(0)$. Since $g([0,1])=[g(1),1]\not\ni0$, we have $z\notin[0,1]$. Choose $u$ so that $z$ is the pole of $\theta_u$. Then $\hat g=\theta_ug\theta_u^{-1}$ satisfies $\hat g(\infty)=0$ and $\hat g(0)=1$, so $\hat g(x)=\delta/(\gamma x+\delta)=a/(x+a)$ with $a=\delta/\gamma$ ($\gamma\neq0$ since $\hat g$ is not constant), and $a>0$ since $\hat g(1)=a/(1+a)\in(0,1)$. (a) is identical with $z=f^{-1}(1)$: then $\hat f(0)=0$, $\hat f(\infty)=1$, so $\hat f(x)=x/(x+a)$, and $a=1/\hat f'(0)=1/f'(0)$.
\end{proof}

Throughout the rest of this section $R$ and $T$ are given by
\begin{equation}\label{eq:RT1}
R(x)=\frac{(ab-c)x+c}{(ab-a-c)x+a+c},\qquad T(x)=\frac{x}{x+a},
\end{equation}
with $a,b>0$ and $c\ge0$. Then $T(0)=0$, $T'(0)=1/a$, $R(1)=1$, $R(0)=c/(a+c)$, $R'(1)=1/b$, and $\det R=a^2b$, $\det T=a$ for the matrices
\[
R=\begin{pmatrix}ab-c&c\\ ab-a-c&a+c\end{pmatrix},\qquad T=\begin{pmatrix}1&0\\1&a\end{pmatrix}.
\]
A direct computation gives
\begin{equation}\label{eq:Rminusx}
R(x)-x=\frac{(1-x)\big((ab-a-c)x+c\big)}{(ab-a-c)x+a+c}.
\end{equation}

For a word $w$ in $R$ and $T$, let $\begin{pmatrix}*&*\\ r_w&s_w\end{pmatrix}$ be the corresponding matrix product. Then $|w'(x)|=\det w/(r_wx+s_w)^2$, and since the denominator is affine in $x$,
\begin{equation}\label{eq:d0d1}
r_wx+s_w\ \ge\ \min\{d_0(w),d_1(w)\}\quad (x\in[0,1]),\qquad d_0(w)=s_w,\ d_1(w)=r_w+s_w.
\end{equation}
Multiplying on the right by $R$ or $T$ one checks
\begin{equation}\label{eq:rec1}
R:\ (d_0,d_1)\mapsto(ad_0+cd_1,\ abd_1),\qquad T:\ (d_0,d_1)\mapsto(ad_0,\ d_1+ad_0),
\end{equation}
starting from $(d_0,d_1)=(1,1)$ for the empty word.

\begin{lemma}\label{lem:word1}
Let $a,b\ge1\ge c>0$. Let $w$ be a word containing $m\ge1$ letters $R$ and $n$ letters $T$, and ending with exactly $k\ge0$ letters $T$. Then for all $x\in[0,1]$
\begin{equation}\label{eq:word1}
|w'(x)|\le c^{-2}a^{\,n-2k+2}\,b^{\,2-m}.
\end{equation}
\end{lemma}

\begin{proof}
By \eqref{eq:rec1}, $d_1$ never decreases under $T$ and is multiplied by $ab$ under $R$, so $d_1\ge(ab)^{j}$ after a prefix containing $j$ letters $R$. Write $w=w_1RT^k$. Immediately after the last $R$ we have $d_0\ge c\,d_1(w_1)\ge c(ab)^{m-1}$ and $d_1=ab\,d_1(w_1)\ge(ab)^m\ge c(ab)^{m-1}$. If $k=0$ we are done; if $k\ge1$, then after $k$ letters $T$, $d_0$ is multiplied by $a^k$ and $d_1\ge a^kd_0^{\rm (before)}$. Hence $\min\{d_0(w),d_1(w)\}\ge ca^{m+k-1}b^{m-1}$, and since $\det w=a^{2m+n}b^m$,
\[
|w'(x)|\le\frac{a^{2m+n}b^m}{c^2a^{2m+2k-2}b^{2m-2}}=c^{-2}a^{n-2k+2}b^{2-m}.\qedhere
\]
\end{proof}

\begin{remark}
Lemma~\ref{lem:word1} replaces Lemma 2.3 of the original version, whose inductive step for words of the form $fR$ did not verify the required inequalities with the new value $k=0$; the estimate \eqref{eq:word1} is what is actually used.
\end{remark}

\begin{proposition}\label{prop:caseI}
Let $a,b\ge1\ge c\ge0$, and suppose $b>1$ if $c=0$. Let $R,T$ be as in \eqref{eq:RT1}. Then the orbit of every $x\in(0,1]$ under $\sgp{R,T}$ is dense in $[0,1]$; if $c>0$ this also holds for $x=0$.
\end{proposition}

\begin{proof}
First let $c=0$, so $b>1$, $R(x)=bx/((b-1)x+1)$ and $T(x)=x/(x+a)$. With $\varphi(x)=(1-x)/(ax)$, $\varphi:(0,1]\to[0,\infty)$, we get $\varphi R\varphi^{-1}(x)=x/b$ and $\varphi T\varphi^{-1}(x)=ax+1$. By Lemma~\ref{lem:affine}(i) every orbit of $\sgp{x/b,ax+1}$ in $[0,\infty)$ is dense, which gives the claim.

Now let $c>0$. By \eqref{eq:Rminusx}, $R(x)>x$ on $[0,1)$ (the factor $(ab-a-c)x+c$ is positive at $0$ and equals $a(b-1)\ge0$ at $1$), so $R^n(0)\to1$. Since $T^n(x)\to0$ for all $x$, it suffices to show that $\Omega=\Omega_0$ equals $[0,1]$. Suppose not, and let $A$ be a maximal-length interval of $(0,1)\setminus\Omega$. The points $0$, $1$, $T(1)=1/(1+a)$ and $R(0)=c/(a+c)\le1/(1+a)$ lie in $\Omega$, so either $A\subseteq(0,1/(1+a))=T((0,1))$ or $A\subseteq(1/(1+a),1)\subseteq\Img(R)$.

In the first case $B=T^{-1}(A)\subseteq(0,1)\setminus\Omega$, and $|A|=\int_B T'<|B|$ because $T'(x)=a/(x+a)^2<1$ for $x>0$; this contradicts maximality. So $A\subseteq(1/(1+a),1)$, and $A_1=R^{-1}(A)\subseteq(0,1)\setminus\Omega$.

\emph{Case 1: $ab-a-c\le0$.} Then $(ab-a-c)x+a+c\ge ab$ on $[0,1]$, so $R'(x)=a^2b/((ab-a-c)x+a+c)^2\le1/b\le1$, with strict inequality for $x<1$ (if $b=1$ then $ab-a-c=-c<0$). Hence $|A|=\int_{A_1}R'<|A_1|$, a contradiction.

\emph{Case 2: $ab-a-c>0$ and $a>1$} (then $b>1$). Every interval in $(0,1)\setminus\Omega$ is contained in a component, which lies in $\Img(R)$ or in $\Img(T)$. Pulling back repeatedly, for each $n$ we obtain a word $w_n$ of length $n$ beginning with $R$ and an interval $A_n\subseteq(0,1)\setminus\Omega$ with $A=w_n(A_n)$; thus $1\le|A|/|A_n|\le\max|w_n'|$. Let $m_n$ and $t_n$ be the numbers of letters $R$ and $T$ in $w_n$. By Lemma~\ref{lem:word1} (and $a>1$),
\begin{equation}\label{eq:I1}
m_n\log_ab\ \le\ t_n+2+2\log_ab-2\log_ac.
\end{equation}
On the other hand $\max R'=R'(0)=a^2b/(a+c)^2$ and $\max T'=1/a$, so
\begin{equation}\label{eq:I2}
t_n\ \le\ m_n\log_a\frac{a^2b}{(a+c)^2}.
\end{equation}
Since $m_n+t_n=n\to\infty$, \eqref{eq:I2} forces $m_n\to\infty$, and then \eqref{eq:I1}, \eqref{eq:I2} give
\[
\log_ab\le\liminf_n\frac{t_n}{m_n}\le\limsup_n\frac{t_n}{m_n}\le\log_a\frac{a^2b}{(a+c)^2}<\log_ab,
\]
a contradiction (recall $c>0$).

\emph{Case 3: $ab-a-c>0$ and $a=1$.} With the notation of Case 2, Lemma~\ref{lem:word1} gives $1\le c^{-2}b^{2-m_n}$, so $m_n$ is bounded. Hence there are a fixed word $h$ and $j_n\to\infty$ with $w_n=hT^{j_n}$. Since $T^{j}(x)=x/(jx+1)$ maps $[0,1]$ onto $[0,1/(j+1)]$,
\[
|A|\le\max|h'|\cdot|T^{j_n}(A_n)|\le\frac{\max|h'|}{j_n+1}\to0,
\]
a contradiction. (The original version claimed $T'\le4/9$ on $(0,1/2)$, which is false since $T'(0)=1$.)
\end{proof}

Case I of Theorem~\ref{thm:main} follows from the next proposition.

\begin{proposition}\label{prop:caseIclass}
Let $f,g\in\F$ be both increasing. Then the orbit of a given $x\in(0,1)$ is dense in $[0,1]$ if and only if one of the following holds:
\begin{enumerate}
\item[(i)] at least one of $f,g$ is not onto, $\Img(f)\cup\Img(g)=[0,1]$, and $\{o(f),o(g)\}=\{0,1\}$;
\item[(ii)] $f(x)=ax/((a-1)x+1)$ and $g(x)=bx/((b-1)x+1)$, where $\min\{a,b\}<1<\max\{a,b\}$ and $\ln a/\ln b\notin\Q$.
\end{enumerate}
\end{proposition}

\begin{proof}
(i) Let $f$ be a map that is not onto. After the change of variable $x\mapsto1-x$ (and renaming) we may assume $o(f)=0$; then $o(g)=1$, $f(0)=0$ and $f(1)<1$. By Lemma~\ref{lem:normal}(a) we may assume $f(x)=x/(x+a)$ with $a=1/f'(0)$; since $o(f)=0$ we have $a\ge1$. The map $g$ is increasing with $g(1)=1$; the increasing maps fixing $1$ form a two-parameter family determined by $g(0)\in[0,1)$ and $g'(1)$, so $g=R$ as in \eqref{eq:RT1} with
\[
c=\frac{a\,g(0)}{1-g(0)}\ge0,\qquad b=\frac1{g'(1)}.
\]
Since $o(g)=1$ we have $b\ge1$, and $b>1$ if $c=0$ (otherwise $g=\mathrm{id}$). The covering condition $g(0)\le f(1)=1/(1+a)$ is equivalent to $c\le1$. The claim follows from Proposition~\ref{prop:caseI}.

(ii) Conjugation by $\theta(x)=1/x-1$ transforms the maps into $x/a$ and $x/b$ on $[0,\infty)$; apply Lemma~\ref{lem:affine}(ii).

\emph{Converse.} Suppose $O(x_0)$ is dense for some $x_0\in(0,1)$. Since $O(x_0)\subseteq\Img(f)\cup\Img(g)$ and the latter is closed, $\Img(f)\cup\Img(g)=[0,1]$.

If $f$ and $g$ are both onto, then conjugating by $\theta(x)=1/x-1$ we obtain maps $x\mapsto kx$ on $(0,\infty)$ (the only increasing linear fractional bijections of $(0,\infty)$), and Lemma~\ref{lem:affine} gives (ii).

If neither is onto, then by the covering condition we may assume $f(0)=0$ and $g(1)=1$. Suppose $o(f)\neq0$. Then $f(x)\ge x$ near $0$; let $q\in(0,1)$ be the smallest positive fixed point of $f$ (it exists since $f(1)<1$). Then $f(y)\ge\min\{y,q\}$ and $g(y)\ge g(0)>0$ for all $y$, so $O(x_0)\subseteq[\min\{x_0,q,g(0)\},1]$, a contradiction. Hence $o(f)=0$ and, symmetrically, $o(g)=1$.

Finally suppose exactly one map, say $f$, is onto. Then $f\neq\mathrm{id}$ (a single map in $\F$ has no dense orbit), and after $x\mapsto1-x$ we may assume $o(f)=0$, i.e.\ $f(x)<x$ on $(0,1)$. If $g(1)<1$ then $O(x_0)\subseteq[0,\max\{x_0,g(1)\}]$. If $g(1)=1$ but $1$ is not attracting for $g$, let $q\in(0,1)$ be the largest fixed point of $g$ in $[0,1)$ (it exists since $g(0)>0$); then $g(y)\le\max\{y,q\}$ and $O(x_0)\subseteq[0,\max\{x_0,q\}]$. Both are impossible, so $o(g)=1$ and (i) holds.
\end{proof}

\section{Case II: mixed monotonicity types}\label{sec:mixed}

Throughout this section
\begin{equation}\label{eq:RT2}
R(x)=\frac{ax}{(-ab+a+c)x+ab},\qquad T(x)=\frac{a}{x+a},
\end{equation}
with $a>0$, $b\ge1$, $0\le c\le1$. Thus $R$ is increasing, $R(0)=0$, $R'(0)=1/b$, $R(1)=a/(a+c)$, and $T$ is decreasing, $T(0)=1$, $T(1)=a/(a+1)$; $\Img(R)\cup\Img(T)=[0,1]$ since $c\le1$. Let $p\in(0,1)$ denote the fixed point of $T$, i.e.\ $p^2+ap-a=0$. Note that
\begin{equation}\label{eq:T2}
T^2(x)=\frac{x+a}{x+a+1},\qquad (T^2)'(x)=\frac1{(x+a+1)^2}<1,\qquad T^2(x)<x\ \text{for } x\in(p,1].
\end{equation}

For words in $R$ and $T$ we again use the denominator values $d_0,d_1$ of \eqref{eq:d0d1}; with the matrices $R=\begin{pmatrix}a&0\\-ab+a+c&ab\end{pmatrix}$, $T=\begin{pmatrix}0&a\\1&a\end{pmatrix}$ one checks
\begin{equation}\label{eq:rec2}
R:\ (d_0,d_1)\mapsto(ab\,d_0,\ cd_0+ad_1),\qquad T:\ (d_0,d_1)\mapsto(ad_1,\ d_0+ad_1),
\end{equation}
and $\det R=a^2b$, $|\det T|=a$. All coefficients are nonnegative.

\begin{lemma}\label{lem:blocks}
Let $a,b,c>0$ and let $R,T$ be as in \eqref{eq:RT2}.
\begin{enumerate}
\item[(a)] For $m\ge1$ and odd $n\ge1$, and $x\in[0,1]$,
\begin{equation}\label{eq:oneblock}
|(R^mT^n)'(x)|\le\frac{ab^m}{(a+cb^{m-1}+b^mx)^2}.
\end{equation}
\item[(b)] For $\alpha_1,\alpha_2\ge1$ and odd $\beta_1,\beta_2\ge1$,
\begin{equation}\label{eq:twoblock}
\max_{[0,1]}\big|(R^{\alpha_1}T^{\beta_1}R^{\alpha_2}T^{\beta_2})'\big|\le\frac{b^{\alpha_2-\alpha_1}}{\big(1+(c^2/a)\,b^{\alpha_2-2}\big)^2}.
\end{equation}
\end{enumerate}
\end{lemma}

\begin{proof}
(a) Write $S_m=1+b+\dots+b^{m-1}\ge b^{m-1}$. By \eqref{eq:rec2}, $R^m$ maps $(1,1)$ to $((ab)^m,\ a^m+ca^{m-1}S_m)$, and then $T$ gives
\[
d_0=a^{m+1}+ca^mS_m\ge a^m(a+cb^{m-1}),\qquad d_1=(ab)^m+a\,d_1^{(R^m)}\ge a^m(a+cb^{m-1}+b^m).
\]
Put $\lambda=a^m$, $A=a+cb^{m-1}$, $B=b^m$, so $d_0\ge\lambda A$ and $d_1\ge\lambda(A+B)$. By \eqref{eq:rec2}, $T^2$ acts by $(d_0,d_1)\mapsto(ad_0+a^2d_1,\ ad_0+(a+a^2)d_1)$, which preserves these inequalities with $\lambda$ replaced by $a\lambda$. Hence after $T^n$, $n=2j+1$, the denominator satisfies $r x+s=(1-x)d_0+xd_1\ge a^{m+j}(A+Bx)$, and since $\det(R^mT^n)=a^{2m+n}b^m$,
\[
|(R^mT^n)'(x)|\le\frac{a^{2m+2j+1}b^m}{a^{2m+2j}(A+Bx)^2}=\frac{ab^m}{(a+cb^{m-1}+b^mx)^2}.
\]
(b) Let $m=\alpha_2$. After the first block, by the proof of (a), $d_0\ge\lambda_1A_1$ and $d_1\ge\lambda_1(A_1+B_1)$ with $\lambda_1=a^{\alpha_1+(\beta_1-1)/2}$, $A_1=a+cb^{\alpha_1-1}$, $B_1=b^{\alpha_1}$. Applying $R^m$ and then $T$ gives
\[
d_0\ge a\big(a^md_1+ca^{m-1}b^{m-1}d_0\big)\ge a^m\lambda_1\big(a(A_1+B_1)+cb^{m-1}A_1\big),
\]
and $d_1\ge d_0$. The remaining $T^{\beta_2-1}$ multiplies these lower bounds by $a^{(\beta_2-1)/2}$. Since $\det=a^{2(\alpha_1+\alpha_2)+\beta_1+\beta_2}b^{\alpha_1+\alpha_2}$ we obtain
\[
|h'(x)|\le\frac{a^2b^{\alpha_1+\alpha_2}}{\big(a(A_1+B_1)+cb^{\alpha_2-1}A_1\big)^2}\le\frac{a^2b^{\alpha_1+\alpha_2}}{\big(ab^{\alpha_1}+c^2b^{\alpha_1+\alpha_2-2}\big)^2},
\]
using $A_1+B_1\ge b^{\alpha_1}$ and $A_1\ge cb^{\alpha_1-1}$. This is \eqref{eq:twoblock}.
\end{proof}

\begin{lemma}\label{lem:longwords}
Let $a,b,c>0$, let $k\ge2$, and let
\begin{equation}\label{eq:fword}
f=R^{\alpha_1}T^{\beta_1}R^{\alpha_2}T^{\beta_2}\cdots R^{\alpha_k}T^{\beta_k},
\end{equation}
where $\alpha_i\ge1$ and all $\beta_i$ are odd. Let $s=\lfloor k/2\rfloor$, $M=\sum_{i\ \mathrm{odd}}\alpha_i$, $N=\sum_{i\ \mathrm{even}}\alpha_i$. If $k$ is even, there exist $u,v\ge0$ and integers $K\ge\max\{M,N\}$, $L\ge N$ with
\begin{equation}\label{eq:even}
|f'(x)|=\frac{a^kb^{M+N}}{\big((ca^{s-1}b^K+a^sb^N+u)x+v+a^sb^M+c^2a^{s-1}b^{L-1}\big)^2}.
\end{equation}
If $k\ge3$ is odd, there exist $u,v\ge0$ and integers $K,L\ge N$ with
\begin{equation}\label{eq:odd}
|f'(x)|=\frac{a^kb^{M+N}}{\big((a^sb^M+c^2a^{s-1}b^L+u)x+v+a^{s+1}b^N+ca^sb^K\big)^2}.
\end{equation}
\end{lemma}

\begin{proof}
We prove \eqref{eq:even} by induction on $s$; \eqref{eq:odd} is proved by the same induction, starting from $k=3$ (obtained by composing the case $k=2$ with one more block). A simple induction on \eqref{eq:rec2} shows that for $m\ge1$ and odd $n$ there are $\gamma,\delta,\mu,\lambda\ge0$ with
\begin{equation}\label{eq:RmTn}
R^mT^n(x)=\frac{a+\gamma+\delta x}{a+cb^{m-1}+b^mx+\mu+\lambda x},\qquad |(R^mT^n)'(x)|=\frac{ab^m}{(a+cb^{m-1}+b^mx+\mu+\lambda x)^2}
\end{equation}
(after dividing numerator and denominator by a common power of $a$). For $k=2$ and $h=R^{\alpha_1}T^{\beta_1}R^{\alpha_2}T^{\beta_2}$ it follows that
\begin{equation}\label{eq:h2}
h(x)=\frac{ab^{\alpha_2}x+cab^{\alpha_2-1}+C_1+C_2x}{(cb^{\alpha_1+\alpha_2-1}+ab^{\alpha_2}+u)x+v+ab^{\alpha_1}+c^2b^{\alpha_1+\alpha_2-2}},
\end{equation}
with $C_1,C_2,u,v\ge0$, and
\[
|h'(x)|=\frac{a^2b^{\alpha_1+\alpha_2}}{\big((cb^{\alpha_1+\alpha_2-1}+ab^{\alpha_2}+u)x+v+ab^{\alpha_1}+c^2b^{\alpha_1+\alpha_2-2}\big)^2},
\]
which is \eqref{eq:even} with $K=L=\alpha_1+\alpha_2-1$. Suppose \eqref{eq:even} holds for $f$ with $k=2s$, and let $g=fh$ with $h$ a two-block word as above (with exponents $\alpha_{k+1},\beta_{k+1},\alpha_{k+2},\beta_{k+2}$). Then
\[
|g'(x)|=|f'(h(x))|\,|h'(x)|=\frac{a^{k+2}b^{M'+N'}}{(Px+Q)^2},
\]
where $M'=M+\alpha_{k+1}$, $N'=N+\alpha_{k+2}$, and, expanding with \eqref{eq:h2}, there are $U,V\ge0$ with
\begin{align*}
P&=ca^sb^{K+\alpha_{k+2}}+a^{s+1}b^{N+\alpha_{k+2}}+ca^sb^{M+\alpha_{k+1}+\alpha_{k+2}-1}+U\ \ge\ ca^sb^{K'}+a^{s+1}b^{N'}+U,\\
Q&=(ca^{s-1}b^K+a^sb^N+u)\,cab^{\alpha_{k+2}-1}+a^{s+1}b^{M+\alpha_{k+1}}+V\ \ge\ c^2a^sb^{L'-1}+a^{s+1}b^{M'}+V,
\end{align*}
with $K'=\max\{K+\alpha_{k+2},M+\alpha_{k+1}+\alpha_{k+2}-1\}\ge\max\{M',N'\}$ and $L'=K+\alpha_{k+2}\ge N'$. This completes the induction.
\end{proof}

Let $\Omega$ denote the closure of the orbit of $0$ under $\sgp{R,T}$. Note that $0$, $1=T(0)$, $a/(a+1)=T(1)$, $a/(a+c)=R(1)$ and $p=\lim T^n(0)$ all lie in $\Omega$.

\begin{lemma}\label{lem:Tfill}
Let $a>0$, $0\le c\le1$, $T(x)=a/(x+a)$, and let $\Omega\subseteq[0,1]$ be closed and $T$-invariant. If $[0,a/(a+c)]\subseteq\Omega$ then $\Omega=[0,1]$.
\end{lemma}

\begin{proof}
Let $w=\max\{t:[0,t]\subseteq\Omega\}\ge a/(a+c)\ge a/(a+1)$ and $v=\min\{t:[t,1]\subseteq\Omega\}$; $v\le T(a/(a+c))=(a+c)/(1+a+c)<1$. Since $T([v,1])=[a/(a+1),T(v)]$ and $a/(a+1)\le w$, we get $[0,\max\{w,T(v)\}]\subseteq\Omega$, so $T(v)\le w$. Since $T([0,w])=[T(w),1]$, we get $v\le T(w)$. If $w<v$, then $T(v)\le w<v$ gives $v>p$, and $T^2(v)\ge T(w)\ge v$, contradicting \eqref{eq:T2}. Hence $w\ge v$ and $\Omega=[0,1]$.
\end{proof}

(The original version stated $T([0,a/(a+c)])=[1/(a+c+1),1]$; the correct image is $[(a+c)/(a+c+1),1]$.)

\begin{lemma}\label{lem:dec1}
Suppose $a,b,c>0$ and $c\le a/(a+c)$. Then for any interval $A\subseteq(0,1)\setminus\Omega$ there are positive integers $\alpha_1,\alpha_2,\ldots$ and $\beta\in\{0,1\}$ such that for each $n\ge1$ there is an interval $A_n\subseteq(0,1)\setminus\Omega$ with
\begin{equation}\label{eq:dec1}
A=T^\beta R^{\alpha_1}TR^{\alpha_2}T\cdots R^{\alpha_n}T(A_n).
\end{equation}
\end{lemma}

\begin{proof}
Since $a/(a+c)\in\Omega$, either $A\subseteq(0,a/(a+c))$ or $A\subseteq(a/(a+c),1)$. In the first case put $\beta=0$ and let $\alpha_1\ge1$ be the largest integer with $B_1=R^{-\alpha_1}(A)\subseteq[0,1]$; it exists because $R^n\to0$ uniformly on $[0,1]$ (Dini's theorem). Then $B_1\not\subseteq\Img(R)=[0,a/(a+c)]$, and $B_1\cap\Omega=\emptyset$, so $B_1\subseteq(a/(a+c),1)\subseteq\Img(T)$. Let $A_1=T^{-1}(B_1)\subseteq(0,c)\subseteq(0,a/(a+c))$, and repeat. In the second case put $\beta=1$ and start with $T^{-1}(A)\subseteq(0,c)$.
\end{proof}

\begin{lemma}\label{lem:dec2}
Suppose $a,b,c>0$, $c\le1$, and $c\ge a/(a+c)$. Then for any interval $A\subseteq(0,a/(a+c))\setminus\Omega$ which is a maximal-length interval of $(0,a/(a+c))\setminus\Omega$, there are positive integers $\alpha_i$ and positive odd integers $\beta_i$ such that for each $n$ there is an interval $A_n\subseteq(0,a/(a+c))\setminus\Omega$ with
\begin{equation}\label{eq:dec2}
A=R^{\alpha_1}T^{\beta_1}R^{\alpha_2}T^{\beta_2}\cdots R^{\alpha_n}T^{\beta_n}(A_n).
\end{equation}
\end{lemma}

\begin{proof}
Let $\alpha_1\ge1$ be largest with $B_1=R^{-\alpha_1}(A)\subseteq[0,1]$; as above $B_1\subseteq(a/(a+c),1)\subseteq\Img(T)$. Let $\beta_1$ be largest with $A_1=T^{-\beta_1}(B_1)\subseteq[0,1]$ (it exists because $T^{2n}\to p$ uniformly). Then $A_1\not\subseteq\Img(T)=[a/(a+1),1]$ and $A_1$ avoids $a/(a+1)\in\Omega$, so $A_1\subseteq[0,a/(a+1))\subseteq[0,a/(a+c)]$; since $0\in\Omega$, $A_1\subseteq(0,a/(a+c))\setminus\Omega$. Iterating gives \eqref{eq:dec2} with $\beta_i\ge1$; it remains to show that each $\beta_i$ is odd.

Suppose $\beta_k$ is the first even exponent and write $F_k=R^{\alpha_1}T^{\beta_1}\cdots R^{\alpha_k}T^{\beta_k}=fT$, where $f$ is a word of the form \eqref{eq:fword} with $k$ blocks and odd exponents. Since $A=F_k(A_k)$ and $|A_k|\le|A|$, we have $\max_{[0,1]}|(fT)'|\ge1$. Note that if $|f'(y)|=a^kb^{M+N}/(Py+Q)^2$ then
\begin{equation}\label{eq:fT}
|(fT)'(x)|=\frac{a^{k+1}b^{M+N}}{\big((PT(x)+Q)(x+a)\big)^2}=\frac{a^{k+1}b^{M+N}}{(Pa+Q(x+a))^2}\le\frac{a^{k-1}b^{M+N}}{(P+Q)^2}.
\end{equation}

\emph{$k=1$.} By \eqref{eq:RmTn} with $m=\alpha_1$ (the exponent normalization there is $a^1$), $P\ge b^m$ and $Q\ge a+cb^{m-1}$, so by \eqref{eq:fT} with $k=1$: $|(fT)'|\le b^m/(a+cb^{m-1}+b^m)^2<b^{-m}\le1$.

\emph{$k\ge2$ even.} By \eqref{eq:even} and \eqref{eq:fT},
\[
|(fT)'(x)|\le\frac{a^{2s+1}b^{M+N}}{\big(ca^sb^K+a^{s+1}b^{\max\{M,N\}}\big)^2}\le\frac{a}{(a+c)^2}\le\frac{c}{a+c}<1,
\]
using $K\ge\max\{M,N\}$ and $a\le c(a+c)$.

\emph{$k\ge3$ odd.} By \eqref{eq:odd} and \eqref{eq:fT},
\[
|(fT)'(x)|\le\frac{a^{k+1}b^{M+N}}{\big(a^{s+1}b^M+c^2a^sb^L+ca^{s+1}b^K\big)^2}.
\]
If $M\ge N$ this is $<a^{2s+2}b^{M+N}/(a^{s+1}b^M)^2=b^{N-M}\le1$. If $N\ge M$, using $K,L\ge N$ it is $<a^2/(c^2+ca)^2=c^{-2}\big(a/(a+c)\big)^2\le1$.

In all cases $\max|(fT)'|<1$, a contradiction.
\end{proof}

\begin{proposition}\label{prop:caseII}
Let $a>0$, $b\ge1\ge c\ge0$, and suppose $b>1$ if $c=0$. Let $R,T$ be as in \eqref{eq:RT2}. Then the orbit of every $x\in[0,1]$ under $\sgp{R,T}$ is dense in $[0,1]$.
\end{proposition}

\begin{proof}
Since $R^n(x)\to0$ and $T(0)=1$, we have $0,1\in\Omega_x$, so it suffices to show $\Omega=[0,1]$. Suppose not.

\emph{The case $c>0$.} If $c\le a/(a+c)$, let $A$ be a maximal-length interval of $(0,1)\setminus\Omega$ and apply Lemma~\ref{lem:dec1}. If $c\ge a/(a+c)$, then by Lemma~\ref{lem:Tfill} $\Omega\not\supseteq[0,a/(a+c)]$, so there is a maximal-length interval $A$ of $(0,a/(a+c))\setminus\Omega$ (recall $0,a/(a+c)\in\Omega$), and we apply Lemma~\ref{lem:dec2} with $\beta=0$. In both cases we obtain $\beta\in\{0,1\}$, $\alpha_i\ge1$ and odd $\beta_i$ with
\begin{equation}\label{eq:Lambda}
\Lambda_k:=\max_{[0,1]}\big|(T^\beta R^{\alpha_1}T^{\beta_1}\cdots R^{\alpha_k}T^{\beta_k})'\big|\ge1\qquad(k\ge1).
\end{equation}
Let $u=\max\{1,1/\sqrt a\}$, so that $\max|(T^\beta)'|\le u^2$. By submultiplicativity and \eqref{eq:twoblock}, for $k=2s$,
\begin{equation}\label{eq:Lam2s}
1\le\Lambda_{2s}\le u^2\prod_{i=1}^s\frac{b^{\alpha_{2i}-\alpha_{2i-1}}}{\big(1+(c^2/a)b^{\alpha_{2i}-2}\big)^2}.
\end{equation}
If $b=1$, the right side is $u^2(1+c^2/a)^{-2s}\to0$, a contradiction. Let $b>1$, and let $M_s=\sum_{i=1}^s\alpha_{2i-1}$, $N_s=\sum_{i=1}^s\alpha_{2i}$. Then \eqref{eq:Lam2s} gives
\begin{equation}\label{eq:NM}
u\,b^{(N_s-M_s)/2}\ge\prod_{i=1}^s\Big(1+\frac{c^2}{ab^2}b^{\alpha_{2i}}\Big)\ge1+\frac{c^2}{ab}\,s,
\end{equation}
so $N_s-M_s\to\infty$. For $k=2s+1$, using \eqref{eq:oneblock} at $x=0$ for the first block and \eqref{eq:twoblock} for the remaining pairs,
\[
1\le\Lambda_{2s+1}\le u^2\frac{ab^{\alpha_1}}{(a+cb^{\alpha_1-1})^2}\prod_{i=1}^s\frac{b^{\alpha_{2i+1}-\alpha_{2i}}}{\big(1+(c^2/a)b^{\alpha_{2i+1}-2}\big)^2},
\]
which gives
\begin{equation}\label{eq:odds}
u\,b^{(\alpha_{2s+1}+M_s-N_s)/2}\ge\sqrt a\prod_{i=1}^s\Big(1+\frac{c^2}{ab^2}b^{\alpha_{2i+1}}\Big)\ge\frac{c^2}{\sqrt a\,b^2}\,b^{\alpha_{2s+1}}.
\end{equation}
Since $N_s-M_s\to\infty$, the left inequality forces $\alpha_{2s+1}\to\infty$; but for large $s$ we have $M_s-N_s\le0$, and then \eqref{eq:odds} gives $b^{\alpha_{2s+1}/2}\le u\sqrt a\,b^2/c^2$, so $\alpha_{2s+1}$ is bounded. This contradiction proves the case $c>0$.

\emph{The case $c=0$} (so $b>1$ and $R(x)=x/((1-b)x+b)$ is onto). Then $R^n(x)=x/((1-b^n)x+b^n)$ and $R^nT(x)=a/(a+b^nx)$, whose fixed point is
\[
\theta_n=\frac{2a}{a+\sqrt{a^2+4ab^n}}\in(0,1).
\]
Since $|(R^nT)'(\theta_n)|=b^n\theta_n^2/a=1-\theta_n<1$, $\theta_n$ is an attracting fixed point of $S_n=(R^nT)^2$. On $[0,\theta_n]$ consider the pair $(R,S_n)$: $R$ is increasing, $R(x)<x$ for $x>0$, $o(R)=0$, $R([0,\theta_n])=[0,R(\theta_n)]$; $S_n$ is increasing, $o(S_n)=\theta_n$, $S_n([0,\theta_n])=[S_n(0),\theta_n]$ with $S_n(0)=R^nT(1)=a/(a+b^n)$. As $n\to\infty$, $S_n(0)\sim ab^{-n}$ while $R(\theta_n)\sim\sqrt a\,b^{-n/2-1}$, so for $n$ large $S_n(0)\le R(\theta_n)$. By Proposition~\ref{prop:caseIclass}(i) (after rescaling $[0,\theta_n]$ to $[0,1]$), the orbits of interior points under $\sgp{R,S_n}$ are dense in $[0,\theta_n]$; since $S_n(0)\in(0,\theta_n)$ the same holds for the orbit of $0$. Finally $R^nT([0,\theta_n])=[\theta_n,1]$, so $\Omega=[0,1]$.
\end{proof}

Case II of Theorem~\ref{thm:main} follows from the next proposition.

\begin{proposition}\label{prop:caseIIclass}
Let $f,g\in\F$ with $f$ increasing and $g$ decreasing. Then the orbit of a given $x\in(0,1)$ is dense in $[0,1]$ if and only if $\{o(f),g(o(f))\}=\{0,1\}$ and one of the following holds:
\begin{enumerate}
\item[(i)] $g$ is not onto and $\Img(f)\cup\Img(g)=[0,1]$;
\item[(ii)] $g$ is onto, $f$ is not onto, and $\Img(fg)\cup\Img(gf)=[0,1]$.
\end{enumerate}
\end{proposition}

\begin{proof}
\emph{Sufficiency of (i).} After $x\mapsto1-x$ we may assume $o(f)=0$, so $g(0)=1$ and $g(1)>0$. By Lemma~\ref{lem:normal}(b) (conjugation by $\theta_u$ preserves all hypotheses) we may assume $g=T$ as in \eqref{eq:RT2}. The increasing maps with $f(0)=0$ form a two-parameter family, and $f=R$ with $b=1/f'(0)$ and $c=a(1/f(1)-1)\ge0$. Since $o(f)=0$, $b\ge1$, and $b>1$ if $c=0$ (if $b=1$ then $f(x)=x/((c/a)x+1)$, which is the identity when $c=0$). The covering condition $f(1)\ge T(1)=a/(a+1)$ is equivalent to $c\le1$. Proposition~\ref{prop:caseII} applies.

\emph{Sufficiency of (ii).} A decreasing linear fractional bijection of $[0,1]$ interchanges $0$ and $1$ and is therefore an involution, so $g^2=\mathrm{id}$. The pair $(f,gfg)$ consists of increasing maps, $f$ is not onto, $o(gfg)=g(o(f))$, $\Img(gfg)=\Img(gf)$ and $\Img(fg)=\Img(f)$. Thus Proposition~\ref{prop:caseIclass}(i) applies, and $\sgp{f,gfg}\subseteq\sgp{f,g}$.

\emph{Necessity.} Suppose $O(x_0)$ is dense, $x_0\in(0,1)$.

Let $g$ be not onto. Clearly $\Img(f)\cup\Img(g)=[0,1]$. If $f$ is onto, then $f\ne\mathrm{id}$, and after $x\mapsto1-x$ we may assume $o(f)=0$; if $g(0)<1$ then $O(x_0)\subseteq[0,\max\{x_0,g(0)\}]$, impossible; so $g(o(f))=1$. If $f$ is not onto, the covering condition gives, after $x\mapsto1-x$, $f(0)=0$ and $g(0)=1$. If $o(f)\neq0$, let $q\in(0,1)$ be the smallest positive fixed point of $f$; then $f(y)\ge\min\{y,q\}$ and $g(y)\ge g(1)>0$, so $O(x_0)\subseteq[\min\{x_0,q,g(1)\},1]$, impossible. Hence $\{o(f),g(o(f))\}=\{0,1\}$.

Let $g$ be onto (so $g^2=\mathrm{id}$) and $f$ not onto. Reducing words using $g^2=\mathrm{id}$, every element of $\sgp{f,g}$ is $g$ or begins with $f$ or with $gf$, hence $O(x_0)\subseteq\{g(x_0)\}\cup\Img(f)\cup\Img(gf)$, and density gives $\Img(f)\cup\Img(gf)=[0,1]$, i.e.\ $\Img(fg)\cup\Img(gf)=[0,1]$. Since $\Img(gf)=[gf(1),gf(0)]$, the condition $0\in\Img(f)\cup\Img(gf)$ forces $f(0)=0$ or $f(1)=1$; after $x\mapsto1-x$ assume $f(0)=0$. Suppose $o(f)\neq0$ and let $q\in(0,1)$ be the smallest positive fixed point of $f$. Choose $m>0$ with $m\le\min\{x_0,g(x_0),q,gf(1)\}$ and let $K=[m,g(m)]$. Then $g(K)=K$; for $y\in K$, $f(y)\ge\min\{y,q\}\ge m$ and $f(y)\le f(1)=g(gf(1))\le g(m)$, so $f(K)\subseteq K$; and $x_0\in K$. Thus $O(x_0)\subseteq K\subsetneq[0,1]$, impossible. Hence $o(f)=0$ and $g(o(f))=1$.

Finally, $f$ and $g$ cannot both be onto. In that case $g^2=\mathrm{id}$, and $gfg$ is an increasing bijection fixing $0$ and $1$ whose multiplier at $0$ is $g'(1)f'(1)g'(0)=f'(1)=1/f'(0)$ (as $g'(0)g'(1)=1$ and $f'(0)f'(1)=1$); since increasing bijections of $[0,1]$ fixing $0,1$ are determined by this multiplier, $gfg=f^{-1}$. Therefore every element of $\sgp{f,g}$ is of the form $f^m$ or $f^mg$ with $m\in\Z$, and the only accumulation points of $O(x_0)$ are $0$ and $1$. (The original version asserted $fg=gf$, which is false unless $f=\mathrm{id}$.)
\end{proof}

\section{Case III: both maps decreasing}\label{sec:dec}

Throughout this section
\begin{equation}\label{eq:RT3}
R(x)=\frac{a(1-x)}{(b-a-c)x+a+c},\qquad T(x)=\frac{a}{x+a},
\end{equation}
with $a>0$, $b\ge1\ge c\ge0$. Both maps are decreasing, $R(1)=0$, $R(0)=a/(a+c)$, $R'(1)=-a/b$, $T(0)=1$, $T(1)=a/(a+1)$, and $\Img(R)\cup\Img(T)=[0,1]$. One computes
\begin{equation}\label{eq:RTcomp}
RT(x)=\frac{ax}{ab+(a+c)x},\qquad RT(1)=\frac{a}{ab+a+c},\qquad (RT)'(0)=\frac1b,
\end{equation}
so $RT(x)<x$ on $(0,1]$ and $o(RT)=0$. Moreover $R\circ(TR)=(RT)\circ R$ and $R$ is a homeomorphism of $[0,1]$ onto $[0,a/(a+c)]$ with $R(1)=0$; hence $(TR)^n(x)\to1$ for $x\in[0,1]$, i.e.\ $o(TR)=1$.

The original version based this case on an estimate (Lemma 4.1) which is false as stated (for instance $|(RT)'(0)|=1/b$ exceeds the asserted bound). We give a different argument, which reduces most of the parameter range to Case II.

\begin{proposition}\label{prop:caseIII}
Let $a>0$ and $b\ge1\ge c\ge0$, and suppose $ab\le1$ if $c=0$. Let $R,T$ be as in \eqref{eq:RT3}. Then the orbit of every $x\in[0,1]$ under $\sgp{R,T}$ is dense in $[0,1]$.
\end{proposition}

\begin{proof}
\emph{Regime (a): $ab+c\le1$.} Consider the pair $(RT,T)$. Here $RT$ is increasing, not onto, $o(RT)=0$, $T(0)=1$, and $\Img(RT)\cup\Img(T)=[0,1]$ because $RT(1)\ge T(1)$ is equivalent to $ab+c\le1$. By Proposition~\ref{prop:caseIIclass}(i) all interior orbits of $\sgp{RT,T}$ are dense, and so are those of $0,1$ (as $T(0)=1$, $T(1)\in(0,1)$).

\emph{Regime (b): $c>0$ and $c\le a/(a+c)$.} Consider the pair $(TR,R)$. Here $TR$ is increasing with $o(TR)=1$, $R$ is decreasing and not onto, $R(o(TR))=R(1)=0$, and $\Img(TR)=[T(a/(a+c)),1]=[(a+c)/(1+a+c),1]$, $\Img(R)=[0,a/(a+c)]$. These cover $[0,1]$ iff $a(1+a+c)\ge(a+c)^2$, i.e.\ iff $c(a+c)\le a$. By Proposition~\ref{prop:caseIIclass}(i) the claim follows.

\emph{Regime (c): $c>a/(a+c)$ and $ab+c>1$.} If $c=0$ then $ab\le1$ and we are in regime (a); so regimes (a)--(c) cover all cases, and here $c>0$. Let $\Omega=\Omega_0$ and suppose $\Omega\neq[0,1]$. Note $0$, $1=T(0)$, $a/(a+1)=T(1)$, $a/(a+c)=R(0)$ and the fixed points of $T$ and $R$ lie in $\Omega$. By Lemma~\ref{lem:Tfill}, $\Omega\not\supseteq[0,a/(a+c)]$; let $A$ be a maximal-length interval of $(0,a/(a+c))\setminus\Omega$.

Let $\alpha\ge1$ be largest with $A_1=R^{-\alpha}(A)\subseteq[0,1]$ (it exists since $R^{2n}$ converges uniformly to the fixed point of $R$). Then $A_1\not\subseteq\Img(R)$ and $A_1$ avoids $R(0)\in\Omega$, so $A_1\subseteq(a/(a+c),1]\subseteq\Img(T)$. Let $\beta\ge1$ be largest with $B=T^{-\beta}(A_1)\subseteq[0,1]$. Then $B$ avoids $T(1)=a/(a+1)$ and $B\not\subseteq\Img(T)$, so $B\subseteq[0,a/(a+1))\subseteq[0,a/(a+c)]$, and $B\cap\Omega=\emptyset$. Thus $A=R^\alpha T^\beta(B)$ with $|B|\le|A|$.

We claim that $|(R^\alpha T^\beta)'|\le1$ on $[0,1]$ with equality at most at finitely many points, for all $\alpha,\beta\ge1$; then $|A|=\int_B|(R^\alpha T^\beta)'|<|B|$, a contradiction. Write
\[
R^\alpha T^\beta=R^{2i}\,X\,T^{2j},\qquad X\in\{RT,\ R^2T,\ RT^2,\ R^2T^2\}.
\]
For a linear fractional map $h$ without pole in $[0,1]$, $|h'|$ is monotone on $[0,1]$, so it suffices to check endpoint values; and a product of factors each $\le1$, with equality only at endpoints, has the required property. Using the matrices of $R$ and $T$ one finds:
\begin{align*}
&(T^2)'(x)=\frac1{(1+a+x)^2}<1,\qquad (R^2)'(0)=\Big(\frac{ab}{ab+c(a+c)}\Big)^2<1,\\
&(R^2)'(1)=\Big(\frac a{a+c}\Big)^2<1,\\
&|(RT)'(0)|=\frac1b\le1,\qquad |(RT)'(1)|=\frac{a^2b}{(ab+a+c)^2}\le\frac b{(b+1)^2}<1,\\
&|(RT^2)'(0)|=\frac{ab}{(ab+a+c)^2}<\frac{ab}{(ab+c)^2}<1,\qquad |(RT^2)'(1)|=\frac{ab}{(ab+a+b+c)^2}<1,\\
&|(R^2T)'(0)|=\frac a{(a+c)^2}<\frac c{a+c}<1,\\
&|(R^2T)'(1)|=\frac{a^3b^2}{(ab+c(a+c)+ab(a+c))^2}<\frac a{(a+1)^2}<1,\\
&|(R^2T^2)'(0)|=\frac{a^2b^2}{(ab+c(a+c)+ab(a+c))^2}<\frac1{(1+a+c)^2},\qquad |(R^2T^2)'(1)|<\frac1{(1+a+c)^2}.
\end{align*}
Here we used $ab+c>1$ (so $(ab+c)^2>ab+c>ab$) and $a<c(a+c)$. This proves the claim. (The original argument used the bounds $|(T^m)'|<1$ and $|(R^m)'|<1$ for all $m>1$, which fail for odd $m$; e.g.\ $|(T^3)'(0)|=1/(a(a+2)^2)$.)
\end{proof}

Case III of Theorem~\ref{thm:main} follows from the next proposition.

\begin{proposition}\label{prop:caseIIIclass}
Let $f,g\in\F$ be both decreasing. Then the orbit of a given $x\in(0,1)$ is dense in $[0,1]$ if and only if $\{o(fg),o(gf)\}=\{0,1\}$ and one of the following holds:
\begin{enumerate}
\item[(i)] neither $f$ nor $g$ is onto and $\Img(f)\cup\Img(g)=[0,1]$;
\item[(ii)] exactly one of $f,g$ is onto and $\Img(fg)\cup\Img(gf)=[0,1]$.
\end{enumerate}
\end{proposition}

\begin{proof}
\emph{Sufficiency of (i).} By the covering condition we may label the maps so that $g(0)=1$ and $f(1)=0$; then $fg(0)=0$, so $o(fg)=0$ and $o(gf)=1$. By Lemma~\ref{lem:normal}(b) we may assume $g=T$. The decreasing maps with $f(1)=0$ form a two-parameter family, and $f=R$ as in \eqref{eq:RT3} with $c=a/f(0)-a$ and $b=a/|f'(1)|$. Since $f$ is not onto, $c>0$; the covering condition $f(0)\ge a/(a+1)$ is $c\le1$; and $o(fg)=0$ gives $(fg)'(0)=|f'(1)|/a=1/b\le1$. Proposition~\ref{prop:caseIII} applies.

\emph{Sufficiency of (ii).} The pair $(fg,gf)$ consists of increasing maps with $\{o(fg),o(gf)\}=\{0,1\}$; if, say, $g$ is onto then $\Img(fg)=\Img(f)\neq[0,1]$. Proposition~\ref{prop:caseIclass}(i) applies.

\emph{Necessity.} Suppose $O(x_0)$ is dense. In case (i), the covering condition is clearly necessary, and we may assume $g(0)=1$, $f(1)=0$, so $fg(0)=0$ and $gf(1)=1$. Suppose $o(fg)\neq0$ and let $q\in(0,1)$ be the smallest positive fixed point of $fg$ (it exists since $fg(1)<f(0)<1$). Let $r\in(0,1)$ be the fixed point of $g$; since $g$ is not an involution, $g^2(x)>x$ on $[0,r)$. Choose $m>0$ so small that $m\le\min\{q,r,x_0\}$, $f(m)\le g(m)$ and $x_0\le g(m)$ (possible since $f(0)<1=g(0)$), and let $K=[m,g(m)]$. Then $f(K)=[f(g(m)),f(m)]\subseteq K$ since $fg(m)\ge m$, and $g(K)=[g^2(m),g(m)]\subseteq K$. Thus $O(x_0)\subseteq K$, a contradiction. Hence $o(fg)=0$, and symmetrically (via $x\mapsto1-x$) $o(gf)=1$.

In case (ii), say $g$ is onto, so $g^2=\mathrm{id}$. Every element of $\sgp{f,g}$ is $g$ or begins with $f$ or $gf$, so $O(x_0)\subseteq\{g(x_0)\}\cup\Img(f)\cup\Img(gf)$ and the covering condition follows ($\Img(fg)=\Img(f)$). The covering condition forces $f(1)=0$ (as $gf(1)=0$ would mean $f(1)=1$). If $o(fg)\neq0$, let $q$ be the smallest positive fixed point of $fg$, choose $m>0$ small with $m\le\min\{q,x_0,g(x_0)\}$ and $gf(m)\ge m$ (possible since $gf(0)=g(f(0))>0$), and let $K=[m,g(m)]$. Then $g(K)=K$, $f(g(m))\ge m$ and $f(m)\le g(m)$, so $f(K)\subseteq K$, a contradiction. Hence $o(fg)=0$ and $o(gf)=g(o(fg))=1$.

Finally, if both $f$ and $g$ are onto, both are involutions and $h=fg$ satisfies $gf=h^{-1}$; every element of $\sgp{f,g}$ is of the form $h^m$ or $h^mf$, $m\in\Z$, and $O(x_0)$ accumulates only at $0$ and $1$.
\end{proof}

\section{Weak mixing; affine pairs on the real line}\label{sec:mix}

We first record the relation between transitivity and hypercyclicity (see also \cite[Prop.~1]{CKN}).

\begin{lemma}\label{lem:baire}
Let $X$ be a second countable Baire space and $G$ a semigroup acting by continuous maps on $X$. Let $\{B_i\}$ be a countable base of nonempty open sets.
\begin{enumerate}
\item[(i)] If $G$ acts by homeomorphisms, then the action of $G$ is topologically transitive iff the action of $G^{-1}=\{h^{-1}:h\in G\}$ is.
\item[(ii)] If the action of $G$ is topologically transitive, then the set of points with dense $G$-orbit is a dense $G_\delta$ subset of $X$.
\item[(iii)] (Transfer) Let $H$ be a semigroup of continuous open maps $X\to X$ acting transitively on $X$, and let $G$ be a semigroup of maps of a set $Y\supseteq X$ such that for every $h\in H$ there is $\hat h\in G$ with $\hat h(h(y))=y$ for all $y\in X$. Then there is a dense $G_\delta$ set of $x\in X$ whose $G$-orbit meets every nonempty open subset of $X$.
\end{enumerate}
\end{lemma}

\begin{proof}
(i) is immediate from the definition. (ii) The sets $O_i=\bigcup_{g\in G}g^{-1}(B_i)$ are open, and dense by transitivity (for open $U\ne\emptyset$ some $g$ has $g(U)\cap B_i\neq\emptyset$). By the Baire property $W=\bigcap_iO_i$ is a dense $G_\delta$; for $w\in W$ and each $i$ there is $g$ with $g(w)\in B_i$. (iii) The sets $O_i=\bigcup_{h\in H}h(B_i)$ are open since each $h$ is open, and dense by transitivity of $H$. If $x\in\bigcap_iO_i$ then for every $i$, $x=h(y)$ with $y\in B_i$, and $\hat h(x)=y\in B_i$.
\end{proof}

\begin{remark}
In the original version, part (ii) assumed in addition that the elements of $G$ are one-to-one and open; this is unnecessary.
\end{remark}

\begin{proof}[Proof of Theorem~\ref{thm:nomix}]
By Remark~\ref{rem:conj} we may assume that the interior of $I$ is $(0,1)$. For $x,y\in(0,1)$ let
\[
d(x,y)=\left|\ln\frac{x(1-y)}{y(1-x)}\right|
\]
be the Hilbert metric of $(0,1)$. For an interval $J=[p,q]$ and $x,y\in(p,q)$ put $d_J(x,y)=|\ln\frac{(x-p)(q-y)}{(y-p)(q-x)}|$. Let $f\in\F_I$ and $J=f(\ol I)$. Since $f$ is a linear fractional bijection of $\ol I$ onto $J$, which maps endpoints to endpoints, invariance of the cross ratio gives $d_J(f(x),f(y))=d(x,y)$. Moreover, for $p'\le p<x<y<q\le q'$, the function $t\mapsto(y-t)/(x-t)$ is increasing on $t<x$ and $t\mapsto(t-x)/(t-y)$ is decreasing on $t>y$, so $d_{[p',q']}\le d_{[p,q]}$. Since $J\subseteq[0,1]$, we conclude
\[
d(f(x),f(y))\le d(x,y)\qquad(x,y\in(0,1)),
\]
and $f((0,1))\subseteq(0,1)$. Hence every $h\in G$ does not increase $d$. The sets $U=\{(x,y)\in(0,1)^2:d(x,y)<1\}$ and $V=\{(x,y)\in(0,1)^2:d(x,y)>2\}$ are nonempty and open in $I^2$, and $h(U)\cap V=\emptyset$ for all $h\in G$. Thus the action on $I^2$ is not topologically transitive.
\end{proof}

\begin{remark}
Theorem~\ref{thm:nomix} strengthens Proposition 5.2 of the original version, which treated semigroups generated by two maps by a case analysis (whose Case 4 also contained an error: an increasing bijection of $[0,1]$ with $o(f)=0$ is not length-decreasing, since $f'(1)=1/f'(0)>1$).
\end{remark}

On the real line, the action of a semigroup generated by a pair of affine maps can be weakly mixing. We need the following lemma.

\begin{lemma}\label{lem:dense2}
Let $a,b\in\R$ be not both nonnegative, $0<|a|<1<|b|$, and $\ln|a|/\ln|b|\notin\Q$. Let $(c_1,c_2)\in\R^2$ and
\[
f(x,y)=(ax,ay),\qquad g(x,y)=(bx+c_1,by+c_2).
\]
Then the orbit of every $(v_1,v_2)\in\R^2$ with $v_1c_2\neq v_2c_1$ under $\sgp{f,g}$ is dense in $\R^2$.
\end{lemma}

\begin{proof}
Let $w=v+(c_1,c_2)/(b-1)$. Then $w\neq0$, since otherwise $v_1c_2=v_2c_1$; interchanging coordinates if necessary, assume $w_1\neq0$ and let $\mu=w_2/w_1$. Let $\Omega$ be the orbit closure of $v$ and $\Lambda$ the (closed) set of $\alpha\in\R$ such that the line $\ell_\alpha=\{y=\mu x+\alpha\}$ is contained in $\Omega$. For $m,n\in\N$,
\[
f^mg^n(z)=a^mb^nz+a^m\frac{b^n-1}{b-1}(c_1,c_2).
\]
The set $\{a^mb^n:m,n\ge N\}$ is dense in $\R$ for every $N$: its absolute values are dense in $(0,\infty)$ since $\ln|a|/\ln|b|\notin\Q$, and if $a<0$ (resp.\ $b<0$) both signs occur, because $2\ln|a|/\ln|b|$ (resp.\ $\ln|a|/2\ln|b|$) is also irrational. Hence for every $t\in\R$ there are $m_i,n_i\to\infty$ with $a^{m_i}b^{n_i}\to t$, and then $f^{m_i}g^{n_i}(v)\to tw$. Thus $\ell_0=\R w\subseteq\Omega$, i.e.\ $0\in\Lambda$. Since $f(\ell_\alpha)=\ell_{a\alpha}$ and $g(\ell_\alpha)=\ell_{b\alpha+\kappa}$ with $\kappa=c_2-\mu c_1=(c_2v_1-c_1v_2)/w_1\neq0$, the set $\Lambda$ is invariant under $\hat f(\alpha)=a\alpha$ and $\hat g(\alpha)=b\alpha+\kappa$. As $\hat f^m\hat g^n(0)=\kappa a^m(b^n-1)/(b-1)$, the same argument shows that $\Lambda\supseteq\{\kappa t/(b-1):t\in\R\}=\R$. Hence $\Omega=\R^2$.
\end{proof}

\begin{theorem}\label{thm:affine-trans}
Let $f,g$ be nonconstant affine maps of $\R$. The action of $\sgp{f,g}$ on $\R$ is topologically transitive if and only if, up to order and after conjugation by an affine map, one of the following holds:
\begin{enumerate}
\item[(i)] $f(x)=x+1$ and $g(x)=x+u$ with $u<0$, $u\notin\Q$;
\item[(ii)] $f(x)=ax$ and $g(x)=bx+c$ with $c\in\{0,1\}$, $0<|a|<1<|b|$, $a,b$ not both positive, and, if $c=0$, $\ln|a|/\ln|b|\notin\Q$;
\item[(iii)] $f(x)=ax$ and $g(x)=x+1$ with $a\in(-\infty,-1)\cup(-1,0)$.
\end{enumerate}
The action is weakly topologically mixing if and only if (ii) holds with $c=1$ and $\ln|a|/\ln|b|\notin\Q$. The action is never topologically $k$-transitive for $k\ge3$.
\end{theorem}

\begin{proof}
\emph{Sufficiency.} (i) $\{m+nu\}$ is dense in $\R$, so every orbit is dense. (iii) If $-1<a<0$, by Lemma~\ref{lem:affine}(i) every orbit of $\sgp{f^2,g}=\sgp{a^2x,x+1}$ of a point $\ge0$ is dense in $[0,\infty)$; every orbit contains a point $\ge0$, and $f$ maps $[0,\infty)$ onto $(-\infty,0]$, so every orbit of $\sgp{f,g}$ is dense. If $a<-1$, apply this to $f^{-1}(x)=x/a$, $g^{-1}(x)=x-1$ (conjugated by $x\mapsto-x$) and use Lemma~\ref{lem:baire}(i).

(ii) with $c=0$: the orbit $\{a^mb^nx\}$ of $x\ne0$ is dense by the argument in Lemma~\ref{lem:dense2}. (ii) with $c=1$: if $a<0<b$, every orbit of $\sgp{a^2x,bx+1}$ of a point $\ge0$ is dense in $[0,\infty)$ by Lemma~\ref{lem:affine}(i); every orbit of $\sgp{f,g}$ contains a point $\ge0$, and $f$ maps $[0,\infty)$ onto $(-\infty,0]$. If $b<0$ (and $a$ of either sign), then $g^2(x)=b^2x+(b+1)$ with $b+1<0$; conjugating by $x\mapsto-x$ and applying Lemma~\ref{lem:affine}(i) to $\sgp{a^2x,b^2x+|b+1|}$, every orbit of $\sgp{f^2,g^2}$ of a point $\le0$ is dense in $(-\infty,0]$. Every orbit of $\sgp{f,g}$ contains a point $\le0$ (apply $g^2f^{2n}$ with $n$ large), and $g$ maps $(-\infty,0]$ onto $[1,\infty)$, whose images under $f^{2n}$ fill $(0,\infty)$. Hence every orbit is dense.

\emph{Necessity.} If $f,g$ are both translations, after scaling $f(x)=x+1$, $g(x)=x+u$; if $u>0$ then $[0,\infty)$ is invariant, and if $u\in\Q$ every orbit is discrete. If exactly one is a translation we may assume $g(x)=x+1$ and $f(x)=ax$, $a\ne1$; if $a>0$, $[0,\infty)$ is invariant, and if $a=-1$ every orbit is discrete. Otherwise neither is a translation; conjugating so that the fixed point of one map is $0$, and scaling, we get $f(x)=ax$, $g(x)=bx+c$, $c\in\{0,1\}$, $a,b\notin\{0,1\}$, and, possibly after exchanging the roles of $f$ and $g$, $|a|\le|b|$. If $a,b>0$, then $[0,\infty)$ is invariant. If $|b|\ge|a|\ge1$: when $|b|>1$ the closed set $\{|x|\ge1/(|b|-1)\}$ (or $\{|x|\ge1\}$ if $c=0$) is invariant; when $|b|=1$ we have $a=b=-1$ and orbits are discrete. If $|a|\le|b|\le1$: when $|b|<1$ the set $\{|x|\le1/(1-|b|)\}$ is invariant; when $b=-1$ and $|a|<1$ the interval $[-\kappa,\kappa+1]$ with $\kappa=|a|/(1-|a|)$ (or $[-1,1]$ if $c=0$) is invariant; when $a=b=-1$ orbits are discrete. In each case the action is not transitive. The remaining case is $|a|<1<|b|$; if $c=0$ and $\ln|a|/\ln|b|\in\Q$, then $\{|a|^m|b|^n\}\subseteq r^{\Z}$ for some $r>1$ and the orbits are not dense.

\emph{Weak mixing.} By Lemma~\ref{lem:dense2}, in case (ii) with $c=1$ and irrational ratio there is a dense set of points of $\R^2$ with dense orbits, so the action on $\R^2$ is transitive. In case (i), $x-y$ is invariant. In case (ii) with $c=0$, orbits of the diagonal action lie on lines through $0$. In case (ii) with $c=1$ and $\ln|a|/\ln|b|\in\Q$, the quantity $z=x-y$ changes by $z\mapsto az$ and $z\mapsto bz$, so along an orbit $|z|$ takes values in $|z_0|r^{\Z}\cup\{0\}$; choosing open sets $U,V$ on which $\ln|x-y|$ ranges in two short intervals whose difference avoids $\Z\ln r$, no element maps $U$ into $V$. In case (iii), $|x-y|$ is multiplied by $|a|$ under $f$ and preserved by $g$, so it is monotone along orbits and the sets $\{|x-y|<1\}$, $\{|x-y|>2\}$ (in the appropriate order) violate transitivity. Finally, $k$-transitivity for $k\ge3$ is excluded by Proposition~\ref{prop:nk}.
\end{proof}

In \cite{Fel2} Feldman proved that for every $n$ there are $n+1$ diagonal matrices generating a hypercyclic semigroup on $\R^n$, and no $n$ diagonal matrices do \cite{CHM}. The action of the group of diagonal matrices on $\R^n$ is not weakly mixing (cf.\ Problem~\ref{prob:3}). In this direction we have:

\begin{proposition}\label{prop:nk}
Let $G$ be a semigroup of linear (respectively, affine) maps of $\R^n$. Then the action of $G$ is not topologically $(n+1)$-transitive (respectively, $(n+2)$-transitive).
\end{proposition}

\begin{proof}
We treat affine maps; the linear case is similar. Suppose the action on $(\R^n)^{n+2}$ is transitive. By Lemma~\ref{lem:baire}(ii) there is $X=(X_1,\dots,X_{n+2})$ with dense orbit. The $n+1$ linear equations $\sum\alpha_iX_i=0$, $\sum\alpha_i=0$ in $n+2$ unknowns have a nonzero solution $\alpha$. For every affine map $h(z)=Lz+t$ we have $\sum\alpha_ih(X_i)=L\sum\alpha_iX_i+t\sum\alpha_i=0$, so the orbit of $X$ lies in the proper subspace $\{\sum\alpha_iY_i=0\}$, a contradiction.
\end{proof}

Contrary to the case of proper intervals, the action of $\sgp{f,g}$ can be hypercyclic but not transitive on $\R$. We now find all hypercyclic actions generated by pairs of affine maps. We say that a closed interval $K$ is a \emph{core} for $(f,g)$ if $f^{-1}(K)\subseteq K$ and $g^{-1}(K)\subseteq K$.

\begin{theorem}\label{thm:affine-hyp}
The action of the semigroup generated by nonconstant affine maps $f,g$ on $\R$ is hypercyclic if and only if, up to order and after conjugation by an affine map, one of the following holds:
\begin{enumerate}
\item[(i)] $f(x)=x+1$ and $g(x)=x+u$ with $u<0$, $u\notin\Q$;
\item[(ii)] $f(x)=ax$ and $g(x)=bx+c$ with $c\in\{0,1\}$, $a,b\in\R\setminus\{0,\pm1\}$, $\max\{|a|,|b|\}>1$ and $|ab|\le|a|+|b|$; if $c=0$ we require in addition that $a,b$ are not both positive, $\min\{|a|,|b|\}<1$, and $\ln|a|/\ln|b|\notin\Q$;
\item[(iii)] $f(x)=ax$ and $g(x)=x+1$ with $a\in(-\infty,-1)\cup(-1,0)\cup(1,\infty)$;
\item[(iv)] $f(x)=-x$ and $g(x)=bx+1$ with $b\in[-2,-1)\cup(1,2]$.
\end{enumerate}
\end{theorem}

\begin{proof}
\emph{The method.} Suppose $K$ is a core for $(f,g)$ with nonempty interior $K^\circ$. Then $\R\setminus K$ and $\R\setminus K^\circ$ are invariant under $f$ and $g$, so a dense orbit must start in $K^\circ$. If the action of $H=\sgp{f^{-1},g^{-1}}$ on $K$ satisfies the hypotheses of Theorem~\ref{thm:main}, then $H$ acts transitively on $K^\circ$ (every interior orbit is dense), and by Lemma~\ref{lem:baire}(iii) (with $\hat h=h^{-1}\in\sgp{f,g}$) there is $x\in K^\circ$ whose $\sgp{f,g}$-orbit is dense in $K$. Its orbit closure is then a closed invariant set containing $\bigcup_{h}h(K)$, and in each case below this union is dense in $\R$. Conversely, if $f^{-1}(K)$ and $g^{-1}(K)$ each contain one endpoint of $K$ and are disjoint, let $\Gamma\ne\emptyset$ be the open gap between them. At each point of $K^\circ$ at most one of $f,g$ keeps it in $K$, so $O(x_0)\cap K^\circ$ is a single (finite or infinite) sequence, and once it enters $\Gamma$ it leaves $K$ for good; hence $|O(x_0)\cap\Gamma|\le1$ and $O(x_0)$ is not dense.

\emph{Reductions.} If both maps are translations, (i) is necessary and sufficient as in Theorem~\ref{thm:affine-trans}. If exactly one is a translation, we may take $g(x)=x+1$, $f(x)=ax$. For $a<0$, $a\ne-1$, the action is transitive (Theorem~\ref{thm:affine-trans}). For $a>1$, $K=(-\infty,0]$ is a core; $f^{-1}(x)=x/a$ maps $K$ onto $K$ with $o(f^{-1})=0$, $g^{-1}(x)=x-1$ is not onto with $o(g^{-1})=-\infty$, so case I(i) of Theorem~\ref{thm:main} applies, and $\bigcup_ng^n(K)=\R$. For $0<a<1$ the orbit of $x_0$ lies in $[\min\{x_0,0\},\infty)$, and for $a=-1$ orbits are discrete. If neither map is a translation we have, as before, $f(x)=ax$, $g(x)=bx+c$, $a,b\ne1$. If $c=0$ the orbits are $\{a^mb^nx\}$, which are dense iff the conditions in (ii) hold. Let $c=1$. If $a=-1$ we are in (iv); if $b=-1$, conjugating $g$ to $x\mapsto-x$ and exchanging the roles of $f,g$ again leads to (iv). Assume $a,b\notin\{0,\pm1\}$. If $\max\{|a|,|b|\}<1$, all orbits are bounded. If $|a|<1<|b|$ or $|b|<1<|a|$, then $|ab|<|a|+|b|$, and exchanging roles we may assume $|a|<1<|b|$; if $a,b$ are not both positive, the action is transitive by Theorem~\ref{thm:affine-trans}. It remains to treat $0<a<1<b$ and $|a|,|b|>1$.

Let $p=1/(1-b)$ denote the fixed point of $g$, $F=f^{-1}(x)=x/a$ and $G=g^{-1}(x)=(x-1)/b$.

\emph{$0<a<1<b$.} $K=(-\infty,p]$ is a core ($p<0$): $F$ maps $K$ into $(-\infty,p/a]\subsetneq K$ with $o(F)=-\infty$, and $G$ maps $K$ onto $K$ with $o(G)=p$. Case I(i) of Theorem~\ref{thm:main} applies. Moreover $f^n(K)=(-\infty,a^np]$ and $g^m$ then fill $\R$.

\emph{$a,b>1$.} $K=[p,0]$, both $F,G$ increasing, $o(F)=0$, $o(G)=p$, $F(K)=[p/a,0]$, $G(K)=[p,-1/b]$. These cover $K$ iff $p/a\le-1/b$, iff $ab\le a+b$ (Theorem~\ref{thm:main} I(i)); otherwise they are disjoint and the gap argument applies. Moreover $\bigcup g^n(K)=[p,\infty)$ and $\bigcup f^n([p,\infty))=\R$.

\emph{$a<-1<1<b$.} $K=[p,p/a]$, $F$ decreasing with $F(K)=[p/a^2,p/a]$, $G$ increasing with $o(G)=p$, $F(p)=p/a$. The images cover iff $G(p/a)\ge p/a^2$, iff $b\le a/(a+1)$, iff $-ab\le b-a$ (case II(i)). Since $0\in K^\circ$ and $|a|>1$, $\bigcup f^n(K)=\R$.

\emph{$a>1$, $b<-1$.} $K=[0,-1/b]$, $o(F)=0$, $G$ decreasing with $G(0)=-1/b$, $F(K)=[0,-1/(ab)]$, $G(K)=[-(1+b)/b^2,-1/b]$. They cover iff $b(1-a)\le a$, iff $-ab\le a-b$ (case II(i)). Moreover $\bigcup f^n(K)=[0,\infty)$ and $g([0,\infty))=(-\infty,1]$.

\emph{$a,b<-1$.} $K=[1/(1-ab),a/(1-ab)]$, the interval between the fixed points of $FG$ and $GF$; both $F,G$ are decreasing and not onto $K$, $F(K)=[1/(1-ab),1/(a(1-ab))]$, $G(K)=[G(a/(1-ab)),a/(1-ab)]$. They cover iff $b(a+1)\le-a$, iff $ab\le-a-b$ (case III(i)). Since $0\in K^\circ$, $\bigcup f^n(K)=\R$.

In all four cases the covering condition is exactly $|ab|\le|a|+|b|$, and when it fails the images are disjoint and contain opposite endpoints of $K$, so the gap argument shows that the action is not hypercyclic.

\emph{Case (iv).} Let $f(x)=-x$, $g(x)=bx+1$. Since $f^2=\mathrm{id}$, $\sgp{f,g}=\sgp{f,gf}$ and $gf(x)=-bx+1$, so we may assume $b>0$ (the original version invoked an affine conjugation here, which does not change the sign of $b$). If $b<1$ orbits are bounded, so let $b>1$ and $q=1/(b-1)$. Then $K=[-q,q]$ is a core; $F=f^{-1}=-x$ is a decreasing bijection of $K$ and $G$ is increasing with $o(G)=-q$ and $F(-q)=q$. By case II(ii) of Theorem~\ref{thm:main}, $H$ is transitive on $K^\circ$ iff $G(K)\cup FG(K)=K$, i.e.\ iff $G(q)=(q-1)/b\ge0$, i.e.\ iff $b\le2$; and then $\bigcup_ng^n(K)\supseteq[-q,\infty)$, so with $f$ we get density in $\R$. If $b>2$ then $q<1$. A dense orbit must start in $K^\circ$; for $y\in K^\circ$, $f(y)=-y\in K^\circ$, $g(y)>1>q$ if $y>0$, $gf(y)>1$ if $y<0$, and for the remaining choice $|g(y)|$ resp.\ $|gf(y)|$ equals $\phi(|y|)$ with $\phi(t)=|1-bt|$. Hence $\{|w|:w\in O(x_0)\cap K^\circ\}\subseteq\{\phi^n(|x_0|)\}$. If $t\in\Gamma=(0,(1-q)/b)$ then $\phi(t)>q$; so at most one term of this sequence lies in $\Gamma$, and $O(x_0)$ meets $\{w:|w|\in\Gamma\}$ in at most two points. Hence the action is not hypercyclic.
\end{proof}

\section{Linear fractional pairs on the real line}\label{sec:mobR}

Every $h\in\F$ is a homeomorphism of $\RP$ and a homeomorphism of $\R$ minus one point onto $\R$ minus one point. For a semigroup $G\subseteq\F$ we say that its action on $\R$ is topologically $k$-transitive if for all nonempty open $U,V\subseteq\R^k$ there is $h\in G$ with $h(U\cap\mathrm{dom}(h)^k)\cap V\neq\emptyset$. Since $\R^k$ is open and dense in $(\RP)^k$, this is equivalent to topological transitivity of the diagonal action on $(\RP)^k$.

\begin{theorem}\label{thm:3trans}
\begin{enumerate}
\item[(a)] No semigroup $G\subseteq\F$ acts topologically $4$-transitively on $\R$.
\item[(b)] Let $t\in(0,\pi)$ with $t/\pi\notin\Q$, and let
\[
f(x)=\frac{x\cos t-\sin t}{x\sin t+\cos t},\qquad g(x)=1-x.
\]
Then the semigroup $\sgp{f,g}$ is dense in $\F\cong PGL_2(\R)$, and its action on $\R$ is topologically $k$-transitive for $k=1,2,3$. In particular it is hypercyclic and weakly topologically mixing.
\end{enumerate}
\end{theorem}

\begin{proof}
(a) The cross ratio $\chi(x_1,x_2,x_3,x_4)=\frac{(x_1-x_3)(x_2-x_4)}{(x_1-x_4)(x_2-x_3)}$ is invariant under every $h\in\F$, so the open sets $U=\{2<\chi<3\}$ and $V=\{4<\chi<5\}$ in $\R^4$ satisfy $h(U)\cap V=\emptyset$ for all $h$.

(b) The map $f$ is elliptic: it is the rotation of $\RP$ (identified with a circle) by the angle $2t$, so $\{f^n\}$ is dense in the compact group $K_i\cong SO(2)/\{\pm1\}$ of elements of $PSL_2(\R)$ fixing $i$ in the upper half-plane. In particular $f^{n_j}\to\mathrm{id}$ for some $n_j\to\infty$, so $f^{-1}=\lim f^{n_j-1}$ lies in the closure $\ol S$ of $S=\sgp{f,g}$; also $g^{-1}=g$. Since $\ol S$ is a closed semigroup containing $f^{\pm1}$ and $g^{\pm1}$, it contains the closure of the group generated by $f$ and $g$, which contains $K_i$ and $gK_ig=K_{1+i}$ (the stabilizer of $1+i$, as $g$ acts on the upper half-plane by $z\mapsto1-\bar z$). The closed subgroup generated by $K_i$ and $K_{1+i}$ is a Lie subgroup whose Lie algebra contains two linearly independent elliptic elements of $\mathfrak{sl}_2(\R)$; the proper subalgebras of $\mathfrak{sl}_2(\R)$ of dimension $\ge2$ are conjugate to the upper triangular subalgebra, which contains no elliptic elements. Hence this subgroup contains $PSL_2(\R)$, and together with the orientation reversing map $g$ we get $\ol S=PGL_2(\R)$.

Now $PGL_2(\R)$ acts transitively on ordered triples of distinct points of $\RP$, and continuously. Given nonempty open $U,V\subseteq(\RP)^3$, choose $u\in U$ with distinct coordinates; the set of $h\in PGL_2(\R)$ with $h(u)\in V$ is open and nonempty, hence contains an element of $S$. This proves $3$-transitivity, and $1$- and $2$-transitivity follow by projection. Hypercyclicity follows from Lemma~\ref{lem:baire}(ii) applied on $\RP$ (removing the countably many points whose orbit passes through $\infty$ does not affect density).
\end{proof}

\begin{remark}
Theorem~\ref{thm:3trans} answers affirmatively the second question in Problem 2 of the original version, and contrasts with Proposition~\ref{prop:nk}: affine maps preserve the ratio $(x_1-x_3)/(x_2-x_3)$ and so are never $3$-transitive on $\R$, while linear fractional maps preserve only the cross ratio of four points.
\end{remark}

\section{Conclusion}\label{sec:concl}

\subsection{Summary}
We summarize the results for the interval $[0,\infty)$. Define an equivalence relation $\sim$ on pairs of maps in $\F_{[0,\infty)}$ generated by $(f,g)\sim(g,f)$ and $(f,g)\sim(\theta f\theta^{-1},\theta g\theta^{-1})$, where $\theta(x)=ux^v$ with $u>0$ and $v\in\{-1,1\}$. Transporting the normal forms \eqref{eq:RT1}, \eqref{eq:RT2}, \eqref{eq:RT3} to $[0,\infty)$ (via $x\mapsto ax/(1-x)$ in case I and $x\mapsto a(1-x)/x$ in cases II and III), Propositions~\ref{prop:caseIclass}, \ref{prop:caseIIclass} and \ref{prop:caseIIIclass} and Theorem~\ref{thm:nomix} give:

\begin{theorem}\label{thm:summary}
Let $f,g:[0,\infty)\to[0,\infty)$ be real linear fractional maps. The orbit of a given $x>0$ under $\sgp{f,g}$ is dense in $[0,\infty)$ if and only if, after replacing $(f,g)$ by an equivalent pair, there exist $a>0$, $b\ge1$, $c\ge0$ such that one of the following holds:
\begin{enumerate}
\item[I.] (i) $a\ge1$, $c\le1$, $b>1$ if $c=0$, and $f(x)=\dfrac{x}{x+a}$, $g(x)=bx+c$;\\
(ii) $a,b>1$, $\ln a/\ln b\notin\Q$, and $f(x)=x/a$, $g(x)=bx$.
\item[II.] (i) $c\le1$, $b>1$ if $c=0$, and $f(x)=\dfrac{a}{x+a}$, $g(x)=bx+c$;\\
(ii) $a\ge1$ and $f(x)=a/x$, $g(x)=bx+1$.
\item[III.] $c\le1$, $ab\le1$ if $c=0$, and $f(x)=\dfrac{a}{x+a}$, $g(x)=c+\dfrac{ab}{x}$.
\end{enumerate}
Moreover, the action of any subsemigroup of $\F_{[0,\infty)}$ on $[0,\infty)$ is never weakly topologically mixing.
\end{theorem}

(The condition $c\le1$ in I(i) was missing in the original version; it is the covering condition $\Img(f)\cup\Img(g)=[0,\infty)$.)

\subsection{Open questions}
Our first question is whether the dense orbits of $\sgp{f,g}$ are uniformly distributed. Several notions of uniform distribution, together with results on the orbits of $\sgp{x/a,bx+c}$, are given in \cite{BMS}; we select one of them.

\begin{problem}
Describe all pairs in $\F_I$ with uniformly distributed orbits in the following sense: there is a measure $\mu$ on $I$ such that for almost every $x$ and every continuous compactly supported $F$,
\[
\lim_{n\to\infty}\frac1{|\sgp{f,g}_n|}\sum_{h\in\sgp{f,g}_n}F(h(x))=\int_IF\,d\mu,
\]
where $\sgp{f,g}_n$ is the set of elements of word length $n$.
\end{problem}

Theorem~\ref{thm:main} deals with proper subintervals. By Theorem~\ref{thm:3trans}, the behaviour on the whole line is much richer: the maximal possible degree of transitivity is $3$, and it is attained.

\begin{problem}
For $k=1,2,3$, determine all pairs in $\F$ whose semigroup action on $\R$ is topologically $k$-transitive. In particular, is every topologically transitive pair $(f,g)$ with $f$ or $g$ elliptic of irrational rotation number automatically $3$-transitive?
\end{problem}

By Proposition~\ref{prop:nk}, a semigroup of linear maps of $\R^n$ is never $(n+1)$-transitive. On the other hand, by \cite{J2} for every $n$ there exist two matrices generating a topologically transitive semigroup action on $\R^n$; the construction there is not weakly mixing.

\begin{problem}\label{prob:3}
Given $n,k\in\N$ with $k\le n$, what is the least $m$ for which some semigroup generated by $m$ linear maps acts topologically $k$-transitively on $\R^n$? In particular, what is the least number of $n\times n$ matrices generating a semigroup dense in the set of all $n\times n$ matrices?
\end{problem}

Finally, the methods of Sections~\ref{sec:inc}--\ref{sec:dec} suggest:

\begin{problem}
Characterize the finite sets $\{f_1,\dots,f_r\}\subseteq\F_I$, $r\ge3$, whose generated semigroup acts topologically transitively on $I$. Is it true that such an action is transitive if and only if it has a transitive two-generated subsemigroup of the form $\sgp{h_1,h_2}$ with $h_1,h_2$ words in the $f_i$?
\end{problem}

\end{document}